\documentclass[12pt]{amsart}
\usepackage{graphicx}
\usepackage{amsmath}
\usepackage{amssymb}
\usepackage[usenames,dvipsnames]{color}
\usepackage{pinlabel}

\newtheorem{thm}{Theorem}[section]
\newtheorem{cor}[thm]{Corollary}

\newtheorem{defin}[thm]{Definition}
\newtheorem{lemma}[thm]{Lemma}

\newtheorem{prop}[thm]{Proposition}
\newtheorem{conj}{Conjecture}

\newcommand{\sss}{\mbox{$\sigma$}}

\newcommand{\bdd}{\mbox{$\partial$}}

\def\zed{{\mathbb Z}}

\def\R{{\mathbb R}}
\def\im{\mathop{\rm Im}\nolimits}

\begin{document}

\title{Fibered knots and Property 2R, II}

\author{Robert E. Gompf}
\address{\hskip-\parindent
Robert Gompf\\
Mathematics Department \\
University of Texas at Austin\\
1 University Station C1200\\
Austin TX 78712-0257, USA}
\email{gompf@math.utexas.edu}

\author{Martin Scharlemann}
\address{\hskip-\parindent
        Martin Scharlemann\\
        Mathematics Department\\
        University of California\\
        Santa Barbara, CA USA}
\email{mgscharl@math.ucsb.edu}

\thanks{Research partially supported by National Science Foundation grants. Thanks also to Mike Freedman and Microsoft's Station Q for rapidly organizing a mini-conference on this general topic.}

\date{\today}

\begin{abstract} Following \cite{ST}, a knot $K \subset S^3$ is said to have {\it Property nR} if, whenever $K$ is a component of an n-component link $L \subset S^3$ and some integral surgery on $L$ produces  $\#_{n} (S^{1} \times S^{2})$, there is a sequence of handle slides on $L$ that converts $L$ into a $0$-framed unlink.  The Generalized Property R Conjecture is that all knots have Property nR for all $n \geq 1$.

The simplest plausible counterexample could be the square knot.  Exploiting the remarkable symmetry of the square knot, we characterize all two-component links that contain it and which surger to $\#_{2} (S^{1} \times S^{2})$.  We argue that at least one such link probably cannot be reduced to the unlink by a series of handle-slides, so the square knot probably does not have Property 2R.  This example is based on a classic construction of the first author.

On the other hand, the square knot may well satisfy a somewhat weaker property, which is still useful in $4$-manifold theory.  For the weaker property, copies of canceling Hopf pairs may be added to the link before the handle slides and then removed after the handle slides. Beyond our specific example, we discuss the mechanics of how addition and later removal of a Hopf pair can be used to reduce the complexity of a surgery description.
\end{abstract}

\maketitle

\section{Introduction: Review of Property nR}

We briefly review the background material from \cite{ST}, to which we refer the reader for details.  Recall the famous Property R theorem, proven in a somewhat stronger form by David Gabai \cite{Ga}:

\begin{thm}[Property R] \label{thm:PropR} If $0$-framed surgery on a knot $K \subset S^3$ yields $S^1 \times S^2$ then $K$ is the unknot.
\end{thm}

There is a natural way of trying to generalize Theorem \ref{thm:PropR} to links in $S^3$.  First recall a small bit of $4$-manifold handlebody theory \cite{GS}.  Suppose $L$ is a link in an oriented $3$-manifold $M$ and each component of $L$ is assigned a framing, that is a preferred choice of cross section to the normal circle bundle of the component in $M$. For example, if $M = S^3$, a framing on a knot is determined by a single integer, the algebraic intersection of the preferred cross-section with the longitude of the knot.  (In an arbitrary $3$-manifold $M$ a knot may not have a naturally defined longitude.)  Surgery on the link via the framing is standard Dehn surgery (though restricted to integral coefficients):  a regular neighborhood of each component is removed and then reattached so that the meridian is identified with the cross-section given by the framing.  Associated to this process is a certain $4$-manifold:  attach $4$-dimensional $2$-handles to $M \times I$ along $L \times \{ 1 \}$, using the given framing of the link components.  The result is a $4$-dimensional cobordism, called the {\em trace} of the surgery, between $M$ and the $3$-manifold $M'$ obtained by surgery on $L$. The collection of belt spheres of the $2$-handles constitutes a link $L' \subset M'$ called the dual link; the trace of the surgery on $L \subset M$ can also be viewed as the trace of a surgery on $L' \subset M'$.

The $4$-manifold trace of the surgery on $L$ is unchanged if one $2$-handle is slid over another $2$-handle.  Such a handle slide is one of several moves allowed in the Kirby calculus \cite{Ki1}.  When the $2$-handle corresponding to the framed component $U$ of $L$ is slid over the framed component $V$ of $L$ the effect on the link is to replace $U$ by the band sum $\overline{U}$ of $U$ with a certain copy of $V$, namely the copy given by the preferred cross-section realizing the framing of $V$. If $M = S^3$ and the framings of $U$ and $V$ are given by the integers $m$ and $n$, respectively, then the integer for $\overline{U}$ will be $m+n\pm 2\cdot link(U,V)$ where $link$ denotes the linking number and the sign is $+$ precisely when the orientations of $U$ and the copy of $V$ fit together to orient $\overline{U}$. (This requires us to orient $U$ and $V$, but the answer is easily seen to be independent of the choice of orientations.) The terms {\em handle addition} and {\em subtraction} are used to distinguish the two cases of sign ($+$ and $-$, respectively) in the formula.

If the handle slide is viewed dually, via the description of the $4$-manifold as the trace of surgery on $L' \subset M'$, the picture is counterintuitive:   Let $U' \subset M'$ and $V' \subset M'$ be the dual knots to $U$ and $V$.  Then the link in $M'$ that is dual to $\overline{U} \cup V$ is $U' \cup \overline{V'}$, where $\overline{V'}$ is obtained by a handle-slide of $V'$ over $U'$.

Let $\#_{n} (S^{1} \times S^{2})$ denote the connected sum of $n$ copies of $S^1 \times S^2$.  The Generalized Property R conjecture (see \cite[Problem 1.82]{Ki2}) says this:

\begin{conj}[Generalized Property R] \label{conj:genR} Suppose $L$ is an integrally framed link of $n \geq 1$ components in $S^3$, and surgery on $L$ via the specified framing yields $\#_{n} (S^{1} \times S^{2})$.  Then there is a sequence of handle slides on $L$ that converts $L$ into a $0$-framed unlink.
\end{conj}

The hypothesis immediately implies that all framings and linking numbers are zero.

\bigskip

There is a way of stating the conjecture so that the focus is on knots, rather than links:

\begin{defin} A knot $K \subset S^3$ has {\bf Property nR} if it does not appear among the components of any $n$-component counterexamples to the Generalized Property R conjecture.  That is, whenever $K$ is a component of an n-component link $L \subset S^3$ and some integral surgery on $L$ produces  $\#_{n} (S^{1} \times S^{2})$, then there is a sequence of handle slides on $L$ that converts $L$ into a $0$-framed unlink
\end{defin}

The Generalized Property R conjecture is that the following is true for all $n \geq 1$:

\begin{conj}[Property nR Conjecture] All knots have Property nR.
\end{conj}

From this perspective, Gabai's theorem \ref{thm:PropR} is that all knots have Property 1R.  It is known (\cite[Prop. 2.3]{ST}) that any counterexample to any Property nR Conjecture must be slice (at least topologically).  It is further shown in \cite{ST} that the unknot has Property 2R and, if there is any fibered counterexample to Property 2R, there is a non-fibered counterexample of lower genus.  These considerations suggest that, if one is searching for a potential counterexample to Property nR, the simplest candidate could be the square knot as a potential counterexample to Property~2R, for the square knot is slice, fibered and of genus $2$.


\section{The square knot in links that surger to $\#_{2} (S^{1} \times S^{2})$} \label{sect:squareknot}

The square knot $Q$ is the connected sum of two trefoil knots, $K_R$ and $K_L$, respectively the right-hand trefoil knot and the left-hand trefoil knot.   There are many $2$-component links containing $Q$ so that surgery on the link gives $\#_{2} (S^{1} \times S^{2})$.  Figure \ref{fig:squareknot} shows (by sliding $Q$ over the unknot) that the other component could be the unknot; Figure \ref{fig:squareknot2b} shows (by instead sliding the unknot over $Q$) that the second component could be quite complicated.  In this section we show that, up to handle-slides of $V$ over $Q$, there is a straight-forward description of all two component links $Q \cup V$, so that surgery on $Q \cup V$ gives $\#_{2} (S^{1} \times S^{2})$.

The critical ingredient in the characterization of $V$ is the collection of properties listed below, which apply because $Q$ is a genus two fibered knot.   Here $h$ denotes the monodromy homeomorphism on the fiber $F_-$ of a genus two fibered knot $U$ and $F$ is the closed genus two surface obtained by capping off $\bdd F_-$ with a disk.  Equivalently, $F$ is the fiber of the {\em closed} genus two fibered $3$-manifold obtained by $0$-framed surgery on $U$.

\begin{prop}  \label{prop:genustwo} Suppose $U \subset S^3$ is a genus two fibered knot and $V \subset S^3$ is a disjoint knot.  Then $0$-framed surgery on $U \cup V$ gives $\#_{2} (S^{1} \times S^{2})$  if and only if  after possible handle-slides of $V$ over $U$,
\begin{enumerate}
\item $V$ lies in a fiber of $U$;
\item $h(V)$ can be isotoped to be disjoint from $V$ in $F$ (though not necessarily in $F_-$);
\item $h(V)$ is not isotopic to $V$ in $F$; and
\item the framing of $V$ given by $F$ is the $0$-framing of $V$ in $S^3$.
\end{enumerate}.
\end{prop}

\begin{proof} See   \cite[Corollary 5.4]{ST} \end{proof}

 \begin{figure}[ht!]
 \labellist
\small\hair 2pt
\pinlabel $0$ at 72 89
\pinlabel $0$ at 137 120
\pinlabel $0$ at 281 18
\pinlabel $0$ at 338 120
\pinlabel $0$ at 432 18
\pinlabel $0$ at 432 120
\pinlabel {band sum here} at 108 -11

 \endlabellist
    \centering
    \includegraphics[scale=0.7]{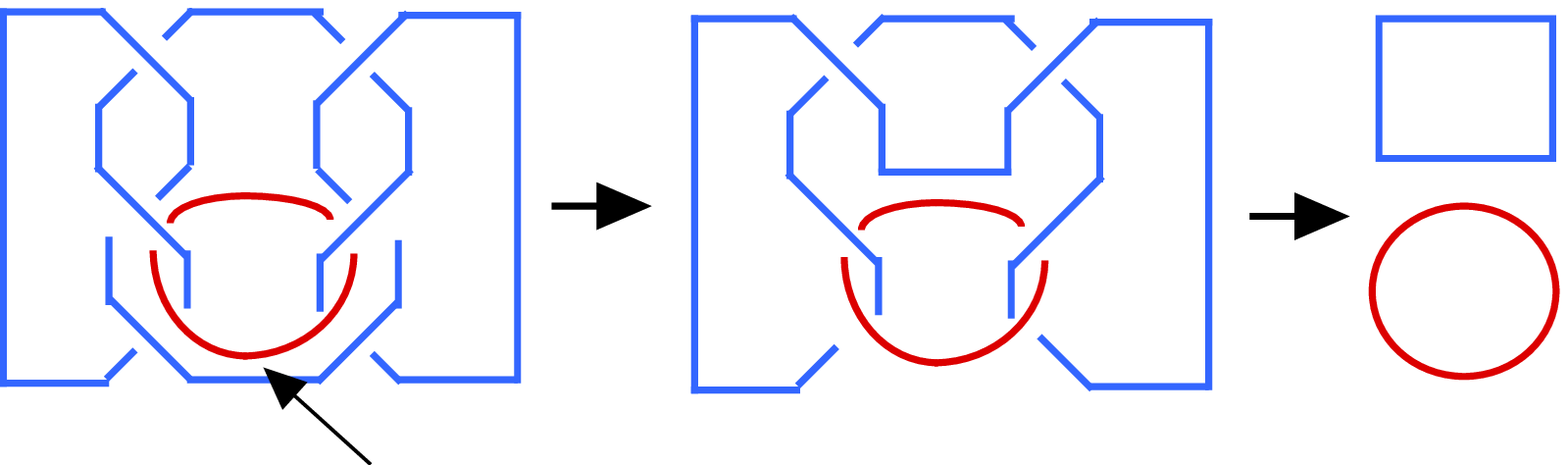}
    \caption{} \label{fig:squareknot}
    \end{figure}

 \begin{figure}[ht!]
    \centering
    \includegraphics[scale=0.7]{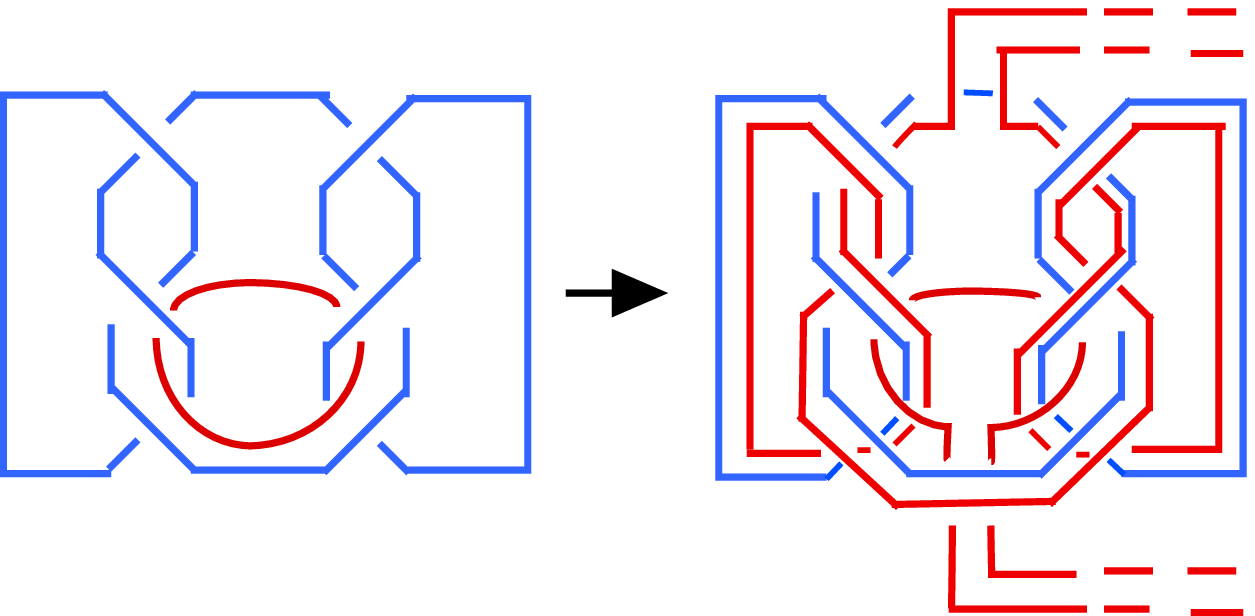}
    \caption{} \label{fig:squareknot2b}
    \end{figure}

Let $M$ be the $3$-manifold obtained by $0$-framed surgery on the square knot $Q$, so $M$ fibers over the circle with fiber the genus $2$ surface $F$.  There is a simple picture of the monodromy $h: F \to F$ of the bundle $M$, obtained from a similar picture of the monodromy on the fiber of a trefoil knot, essentially by doubling it \cite[Section 10.I]{Ro}:

 \begin{figure}[ht!]
  \labellist
\small\hair 2pt
\pinlabel $F$ at 210 700
\pinlabel $P$ at 410 670
\pinlabel \color{red}{$\sss$} at 266 607
\pinlabel \color{ForestGreen}{$\rho$} at 195 770
\pinlabel \color{Brown}$\gamma/\rho$ at 360 660
\pinlabel \color{black}{$/\rho$} at 298 612

 \endlabellist
    \centering
    \includegraphics[scale=0.7]{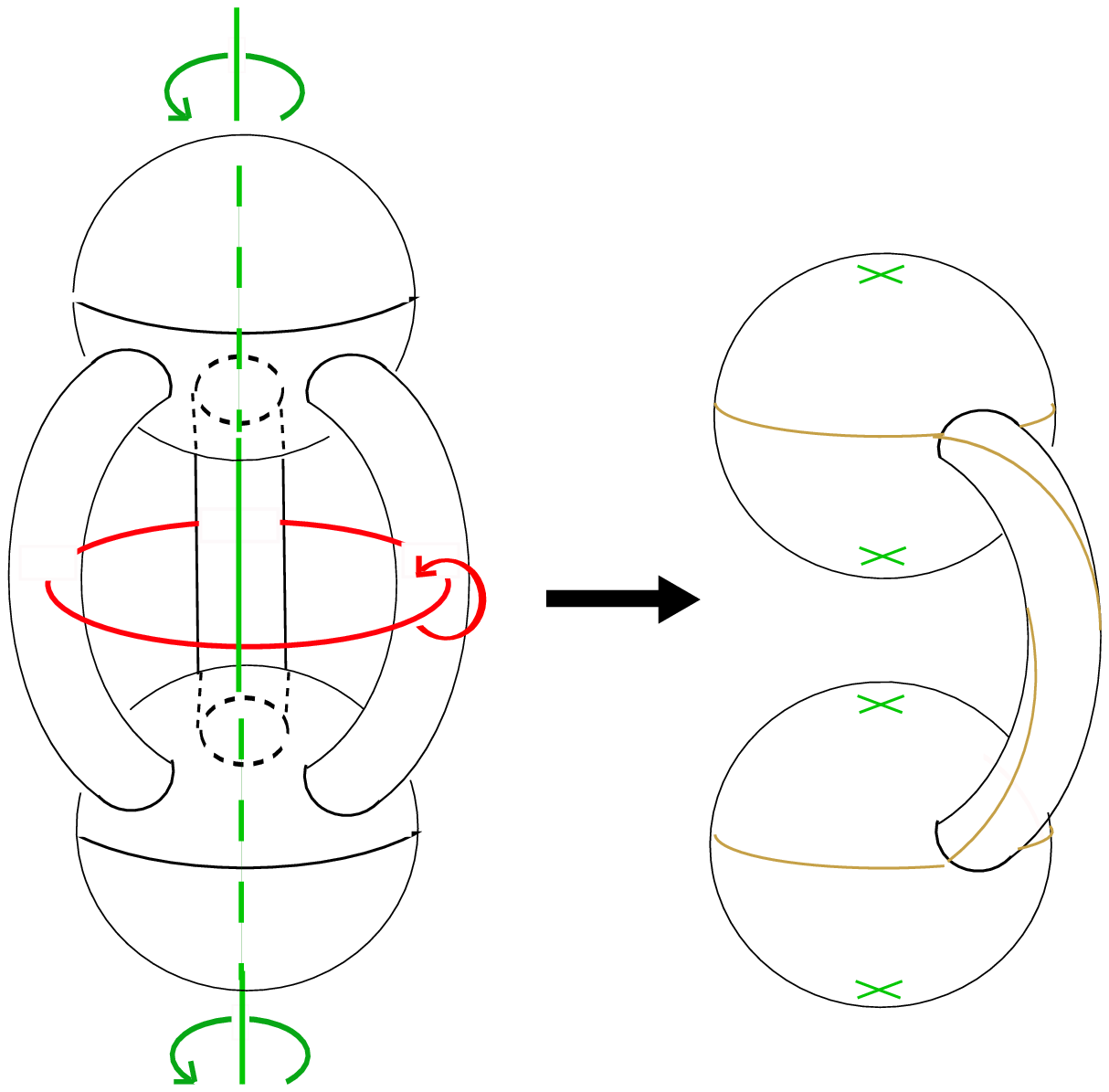}
    \caption{} \label{fig:rhosigma}
    \end{figure}

Regard $F$ as obtained from two spheres by attaching $3$ tubes between them.  See Figure \ref{fig:rhosigma}.  There is an obvious period $3$ symmetry $\rho: F \to F$ gotten by rotating $\frac{2\pi}{3}$ around an axis intersecting each sphere in two points, and a period $2$ symmetry (the hyperelliptic involution) $\sss: F \to F$ obtained by rotating around a circular axis that intersects each tube in two points.  Then $h = \rho \circ \sss = \sss \circ \rho$ is an automorphism of $F$ of order $2 \times 3 = 6$.

The quotient of $F$ under the action of $\rho$ is a sphere with $4$ branch points, each of branching index $3$.  Let $P$ be the $4$-punctured sphere obtained by removing the branch points.   A simple closed curve in $P$ is {\em essential} if it doesn't bound a disk and is not $\bdd$-parallel.  Put another way, a simple closed curve in $P$ is essential if and only if it divides $P$ into two twice-punctured disks.
It is easy to see (viz. \cite[Section 6]{ST}) that there is a separating simple closed curve $\gamma \subset F$ that is invariant under $\sss$ and $\rho$, and hence under $h$, that separates $F$ into two punctured tori $F_R$ and $F_L$; the restriction of $h$ to $F_R$ or $F_L$ is the monodromy of the trefoil knot.   The quotient of $\gamma$ under $\rho$ is shown as the brown curve in Figure \ref{fig:rhosigma} .

An important property of the hyperelliptic involution $\sss: F \to F$ is that it fixes the isotopy class of each curve.  That is, for any simple closed curve $c \subset F$, $\sss(c)$ is isotopic to $c$.  The isotopy preserves orientation on $c$ if and only if $c$ is separating.

This immediately gives
\begin{cor} For any simple closed curve $c \subset F$, $h(c)$ is isotopic in $F$ to a curve disjoint from $c$ if and only if $\rho(c)$ is isotopic in $F$ to a curve disjoint from $c$.
\end{cor}
\begin{proof}  Since $h^3(c) = \sss(c)$ is isotopic to $c$ in $F$, $h(c)$ is isotopic to a curve disjoint from $c$ if and only if $h(c)$ is isotopic to a curve disjoint from $h^3(c)$.  Applying $h^{-3}: F \to F$ shows that $h(c)$ is isotopic to a curve disjoint from $h^3(c)$ if and only if $h^{-2}(c) = \rho(c)$ is isotopic to a curve disjoint from $c$.
\end{proof}

  \begin{lemma} \label{lemma:overline}For the $3$-fold branched covering $/\rho: F \to S^2 \supset P$ described above: \begin{enumerate}
 \item An essential simple closed curve $c \subset F$ has the property that $\rho(c)$ can be isotoped to be disjoint from $c$ if and only if $c$ is isotopic to a lift of an essential simple closed curve $\overline{c}$ in $P$.

 \item  A lift $c$ of an essential simple closed curve $\overline{c}$ in $P$ is separating in $F$ if and only if $\overline{c}$ separates the pair of branch points coming from $F_L$ from the branch points coming from $F_R$.

 \item  A single lift $c$ of an essential simple closed curve $\overline{c}$ in $P$ is non-separating in $F$ if and only if $c$ projects homeomorphically to $\overline{c}$.
 \end{enumerate}
 \end{lemma}

 \begin{proof}  Since $\overline{c}$ is embedded, any lift $c$ of $\overline{c}$ will either be carried by $\rho$ to a disjoint curve, or to $c$ itself.  In either case, $\rho(c)$ can be isotoped off of $c$.  On the other hand, suppose $\rho(c)$ can be isotoped off of $c$.  Give $F$ the standard hyperbolic metric, which is invariant under $\rho$, and isotope $c$ to a geodesic in $F$.  Then $\rho(c)$ is also a geodesic and, since they can be isotoped apart, either $\rho(c) = c$ or $c \cap \rho(c) = \emptyset$.  Apply $\rho^{\pm 1}$ to deduce that  $c, \rho(c), \rho^{2}(c)$ either all coincide or are all disjoint.  In either case, $c$ projects to an embedded curve in $P$.  This proves the first part of the Lemma.

To prove the second and third parts, we first establish:

\bigskip

{\bf Claim:}  An essential simple closed curve $\overline{c} \subset P$ lifts to a curve $c$ that is invariant under $\rho$ if and only if $c$ is separating in $F$.

\bigskip

\noindent {\em Proof of Claim.}  Since $\pm 1$ are not eigenvalues of the monodromy matrix for $h_*: \mathbb{Z}^4 \to \mathbb{Z}^4$, no non-separating curve is left invariant by $h$ or $\rho$.  On the other hand, if the essential curve $c$ (and hence also $\rho(c)$) is separating and $\rho(c)$ is disjoint from $c$ then, since $F$ only has genus $2$, $c$ and $\rho(c)$ must be parallel in $F$.  This proves the claim.

\bigskip

The Claim establishes the third part of the Lemma.  For the second part, we restate the Claim: whether $\overline{c}$ lifts to a non-separating or to a separating curve is completely determined by whether $\overline{c}$ is covered in $F$ by three distinct curves or is thrice covered by a single curve.  That, in turn, is determined by whether $\overline{c}$ represents an element of $\pi_1(P)$ that is mapped trivially or non-trivially to $\mathbb{Z}_3$ under the homomorphism $\pi_1(P) \to \mathbb{Z}_3$ that defines the branched covering $F \to S^2 \supset P$.  One way to view such a curve $\overline{c}$ in $P$ is as the boundary of a regular neighborhood of an arc $\underline{c}$ between two branch points of $S^2 \supset P$, namely a pair of branch points (either pair will do) that lie on the same side of $\overline{c}$ in $S^2$.  From that point of view, $\overline{c}$ represents a trivial element of $\mathbb{Z}_3$ if and only if the normal orientations of the branch points at the ends of the arc $\underline{c}$ disagree; that is, one is consistent with a fixed orientation of the rotation axis and one is inconsistent.  Put another way, $\overline{c}$ represents a trivial element if and only if one endpoint of $\underline{c}$ is a north pole in Figure \ref{fig:rhosigma} and the other is a south pole.  The brown curve in the Figure, which lifts to the curve separating $F_R$ from $F_L$, separates both north poles from both south poles.  This establishes the second part of the Lemma.  \end{proof}

\begin{lemma} \label{lemma:framing} If $c \subset F$ projects to a simple closed curve $\overline{c}$ in $P$ then the framing of $c$ given by the Seifert surface $F_- \subset S^3$ of $Q$ is the $0$-framing. \end{lemma}

\begin{proof}  Here is an equivalent conclusion:  If $c'$ is a parallel copy of $c$ in $F_-$, then $link(c, c') = 0$ in $S^3$.  Consider the following automorphism $\tau$ of $S^3$ that preserves the Seifert surface $F_-$ of $Q$, as viewed in the left side of Figure  \ref{fig:tau}:  Reflect through the plane of the knot projection (this reflection makes the right-hand trefoil a left-hand trefoil, and vice versa), then do a $\pi$ rotation around an axis in $S^3$ that passes through the center of the arc $\gamma_- \subset F_-$ that separated $K_R$ from $K_L$.  The automorphism is orientation preserving on the Seifert surface, but orientation reversing on $\gamma$ and on $S^3$.  Because of the last fact, $link(\tau(c), \tau(c')) = -link(c, c')$.

The automorphism $\tau |F_-$ descends to an automorphism $\overline{\tau}$ of $S^2 \supset P$ given by $\pi$-rotation about the blue axis shown in Figure  \ref{fig:tau}.  This rotation exchanges $F_R$ and $F_L$ but, up to isotopy and orientation, leaves invariant the simple closed curves in $P$.  It follows that in $F$, $\tau(c)$ is isotopic to some $\pm \rho^i(c)$.  Of course the monodromy $h$, hence $\rho$, preserves the linking pairing. (The fiber structure describes an isotopy of a curve to its image under the monodromy.)  Hence we have, for some $0 \leq i  \leq 2$,  $$link(c, c') = -link(\tau(c), \tau(c')) = -link(\pm \rho^i(c), \pm \rho^i(c')) = -link(c, c').$$  This implies that $link(c,  c') = 0$.
\end{proof}

 \begin{figure}[ht!]
  \labellist
\small\hair 2pt
\pinlabel \color{blue}{$\overline{\tau}$} at 285 125
\pinlabel \color{black}{$\gamma_-$} at 190 90
 \endlabellist
    \centering
    \includegraphics[scale=0.8]{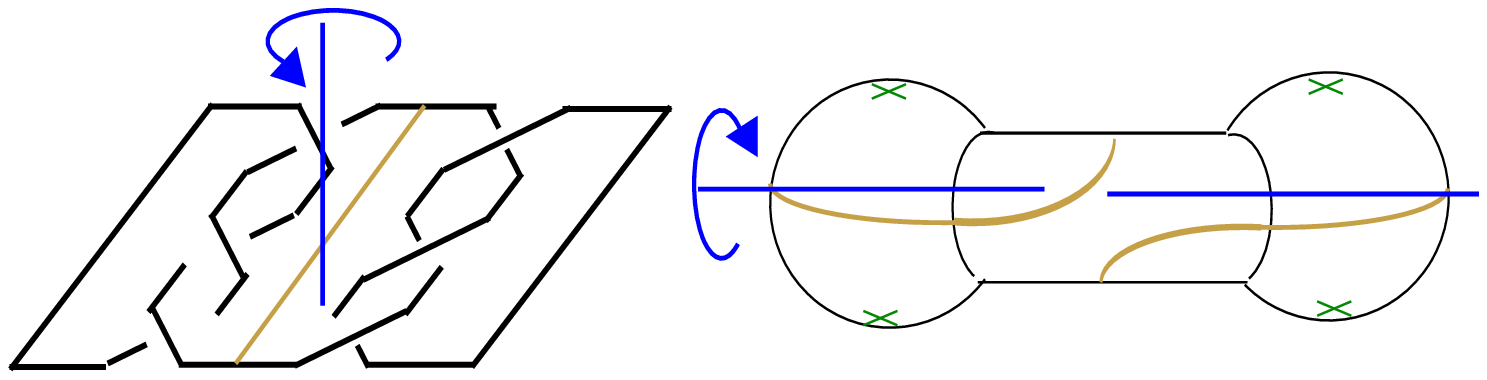}
    \caption{} \label{fig:tau}
    \end{figure}

\begin{cor} \label{cor:enumerate} Suppose $Q \subset S^3$ is the square knot with fiber $F_- \subset S^3$ and $V \subset S^3$ is a disjoint knot.  Then $0$-framed surgery on $Q \cup V$ gives $\#_{2} (S^{1} \times S^{2})$ if and only if, after perhaps some handle-slides of $V$ over $Q$, $V$ lies in $F_-$ and $\rho$ projects $V$ homeomorphically to an essential simple closed curve in $P$.
\end{cor}

Note that, according to Lemma \ref{lemma:overline}, essential simple closed curves $\overline{c}$ in $P$ that are such homeomorphic projections are precisely those for which one branch point of $F_L$ (or, equivalently, one branch point from $F_R$) lies on each side of $\overline{c}$. So another way of saying that $\rho$ projects $V$ homeomorphically to an essential simple closed curve in $P$ is to say that $V$ is the lift of an essential simple closed curve in $P$ that separates one branch point of $F_L$ (or, equivalently $F_R$) from the other.

\begin{proof}  This is an application of Proposition \ref{prop:genustwo}.  Suppose that  $0$-framed surgery on $Q \cup V$ gives $\#_{2} (S^{1} \times S^{2})$.  Then, after possible handle-slides over $Q$, $V$ lies in $F_-$.
According to Proposition \ref{prop:genustwo}, in the closed surface $F$, $h(V)$ can be isotoped to be disjoint from $V$ but is not isotopic to $V$.  The result then follows from Lemma \ref{lemma:overline}.

Conversely, suppose that after some handle-slides of $V$ over $Q$, $V$ lies in $F_-$ and $/\rho:F \to S^2 \supset P$ projects $V$ homeomorphically to an essential simple closed curve $\overline{c}$ in $P$.  $\rho(V)$ then can't be isotopic to $V$ in $F$ (or else in fact $\rho(V) = V$ and so $\rho|V: V \to \overline{c}$ would not be a homeomorphism) but $\rho(V)$ is disjoint from $V$, since $\overline{c}$ is embedded.  Thus the first two conditions of Proposition \ref{prop:genustwo} are satisfied.  Lemma \ref{lemma:framing} establishes the third condition.
\end{proof}

\section{The $4$-manifold viewpoint: a non-standard handle structure on $S^4$}    \label{sect:nonstandard}

To understand why the square knot $Q$ is likely to fail Property 2R, it is useful to ascend to 4 dimensions. We use an unusual family of handlebodies from \cite{Go1} to construct a family of 2-component links in $S^3$, each containing $Q$ as one component. Each of these links will fail to satisfy Generalized Property R if a certain associated presentation of the trivial group fails (as seems likely) to satisfy the Andrews-Curtis Conjecture. Each link also has another intriguing attribute: it is smoothly slice, and each component is ribbon, but it is not known whether the entire link is ribbon.

In \cite{Go1}, the first author provided unexpected examples of handle structures on homotopy $4$-spheres which do not obviously simplify to give the trivial handle structure on $S^4$.  At least one family is highly relevant to the discussion above.  This is example \cite[Figure 1]{Go1}, reproduced here as the left side of Figure \ref{fig:Gompffig1b}.  (Setting $k = 1$ gives rise to the square knot.)  The two circles with dots on them represent $1$-handles, indicating that the $4$-manifold with which we begin is $(S^1 \times D^3) \natural (S^1 \times D^3)$ with boundary $\#_{2} (S^{1} \times S^{2})$ given by $0$-surgery.  (Think of the two dotted unknotted circles as bounding disjoint unknotted disks in the $4$-ball; the dots indicate one should scoop these disks out of $D^4$ rather than attach $2$-handles to $D^4$.)  The circles without dots represent $2$-handles attached to $\#_{2} (S^{1} \times S^{2})$ with the indicated framing.  A sequence of Kirby operations in \cite[\S 2]{Go1} shows that the resulting $4$-manifold has boundary $S^3$.

 \begin{figure}[ht!]
  \labellist
\small\hair 2pt
\pinlabel {\tiny $-n-1$} at 156 758
\pinlabel {\tiny $-n-1$} at 451 758
\pinlabel $n$ at 154 678
\pinlabel $n$ at 448 678
\pinlabel $0$ at 177 718
\pinlabel $-1$ at 117 718
\pinlabel $k$ at 85 718
\pinlabel $-k$ at 232 718
\pinlabel $k$ at 377 718
\pinlabel $-k$ at 527 718
\pinlabel $[0]$ at 470 718
\pinlabel $[0]$ at 336 772
\pinlabel $[0]$ at 336 666
\pinlabel $[-1]$ at 422 705
\pinlabel $0$ at 422 733
\pinlabel $0$ at 508 635
\endlabellist
    \centering
    \includegraphics[scale=0.7]{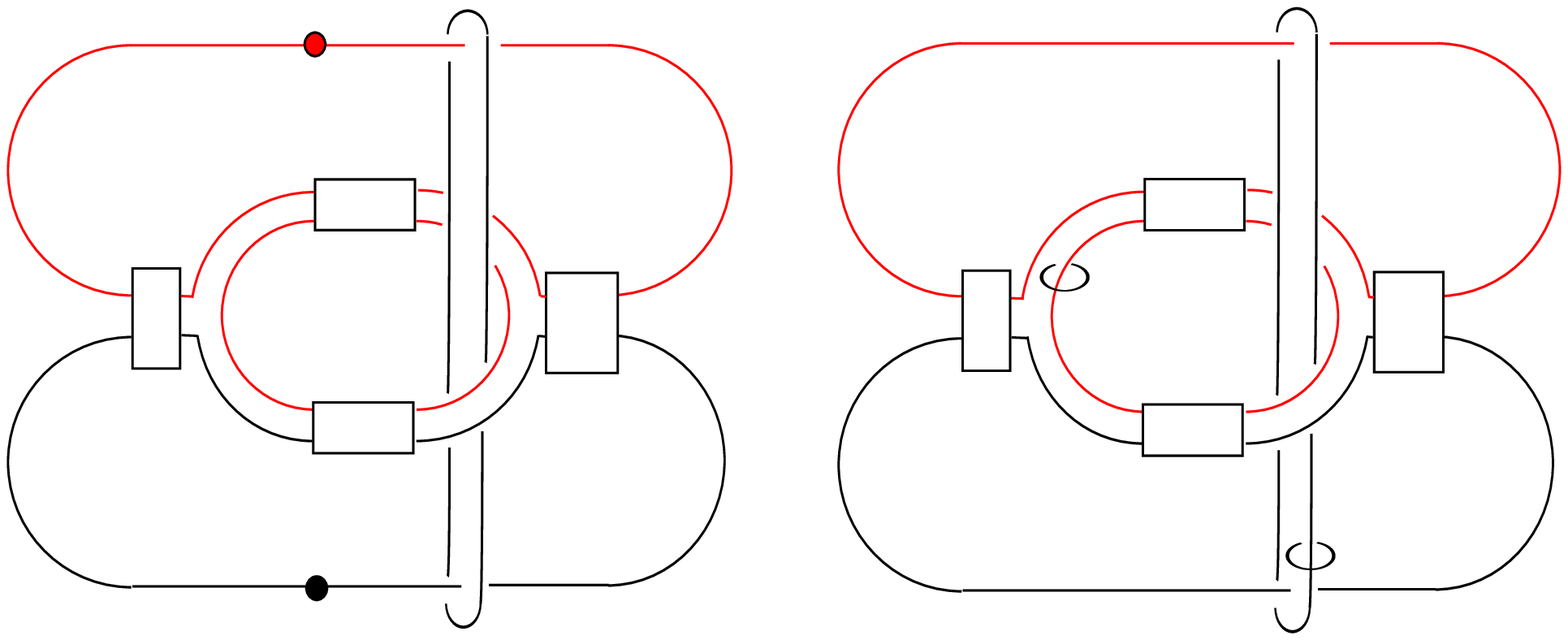}
    \caption{} \label{fig:Gompffig1b}
    \end{figure}

We will be interested in the $4$-manifold that is the trace of the $2$-handle surgeries, the manifold that lies between $\#_{2} (S^{1} \times S^{2})$ and $S^3$.  Our viewpoint, though, is dual to that in \cite{Go1}; the $4$-manifold is thought of as starting with $S^3$ (actually, with $\bdd D^4$) to which two $2$-handles are attached to get $\#_{2} (S^{1} \times S^{2})$.  This dual viewpoint puts the construction solidly in the context of this paper.

A recipe for making the switch in perspective is given in \cite[Example 5.5.5]{GS}.  The result is shown in the right half of Figure \ref{fig:Gompffig1b}.  The circles representing $1$-handles are changed to $0$-framed $2$-handles and then all $2$-handles have their framing integers bracketed.  We call these circles bracketed circles; they form a complicated surgery description of $S^3$.  New $0$-framed $2$-handles are added, each linking one of these original $2$-handles.  These small linking $2$-handles will eventually be the $2$-handles we seek, after we eliminate the other components by following the recipe for turning the original diagram into a trivial diagram of the $3$-sphere.  This process allows  handle slides of bracketed or unbracketed circles over bracketed ones, but does not allow handle slides of bracketed circles over unbracketed circles.  The reduction is done in Figure \ref{fig:Gompffig2b} and roughly mimics  \cite[Figures 7 and 8]{Go1}: The top (red) circle is handle-slid over the bottom (black) circle $C$ via a band that follows the tall thin $[0]$-framed circle.  In order to band-sum to a $0$-framed push-off  of $C$, the $n$ twists that are added in the lower twist box are undone by adding $-n$ twists via a new twist box to the left.  The Hopf link at the bottom of the left diagram is removed in a $2$-stage process: The unbracketed component is handleslid over $C$, so it is no longer linked with the $[0]$-framed component of the Hopf link.  Then $C$ is canceled with the $[0]$-framed circle of the Hopf link.  The result is the picture on the right.  The apparent change is both that the Hopf link has been removed and the framing of $C$ has been changed from $[0]$ to $0$.

 \begin{figure}[ht!]
  \labellist
\small\hair 2pt
\pinlabel {\small $-n-1$} at 150 520
\pinlabel {\small $-n-1$} at 450 520
\pinlabel $n$ at 455 440
\pinlabel $n$ at 150 440
\pinlabel $-n$ at 25 425
\pinlabel $-n$ at 337 425
\pinlabel $0$ at 176 360
\pinlabel $[0]$ at 25 476
\pinlabel $C$ at 275 476
\pinlabel $0$ at 340 476
\pinlabel $[0]$ at 140 360
\pinlabel $[0]$ at 40 530
\pinlabel $[0]$ at 352 530
\pinlabel $[-1]$ at 125 500
\pinlabel $[-1]$ at 433 500
\pinlabel $0$ at 180 495
\pinlabel $0$ at 480 500
\pinlabel $k$ at 78 476
\pinlabel $-k$ at 227 476
\pinlabel $k$ at 380 476
\pinlabel $-k$ at 527 476

\endlabellist
    \centering
    \includegraphics[scale=0.7]{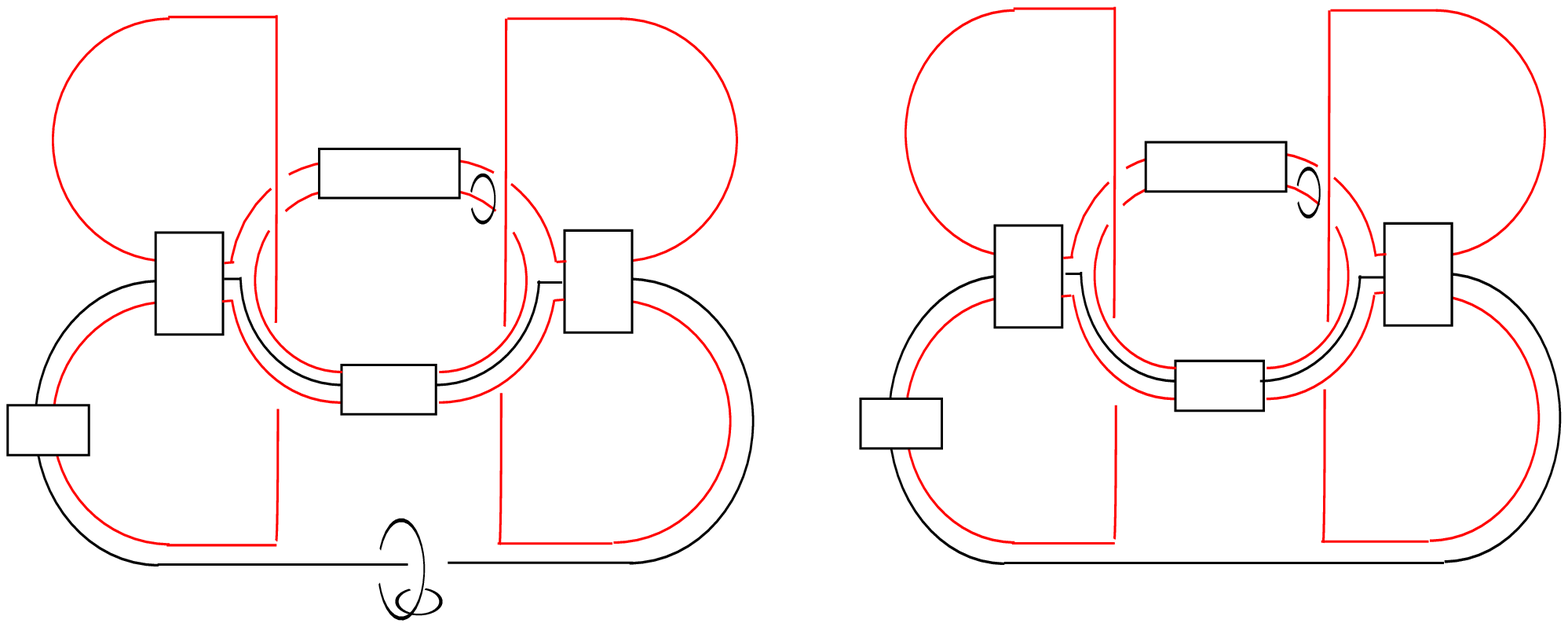}
    \caption{} \label{fig:Gompffig2b}
    \end{figure}

Next (Figure \ref{fig:Gompffig3b}) the large red $[0]$-framed circle is handle-slid over the central red $[-1]$ -framed circle.  That gives the figure on the left.  Note that the twist boxes guaranteee we have pushed off with framing $-1$ as required.  Since there is now only one strand going through the top twist-box, the box can be ignored.  The handle-slide changes the $[0]$-framing on the slid circle to a $[+1]$-framing, since the two curves, if oriented in the same direction through the braid boxes, have linking number $n + (-n-1) = -1$ and, with that orientation, the handle-slide corresponds to a handle subtraction. To get the figure on the right, the lower black circle is moved into position and a flype moves the lower twist-box to a new upper twist-box.  (Note the sphere enclosing the left half of the diagram, intersecting the $n$-twist box and two other strands.)

 \begin{figure}[ht!]
  \labellist
\small\hair 2pt
\pinlabel $n$ at 155 673 %
\pinlabel $n$ at 435 730 %
\pinlabel $-n$ at 35 656 %
\pinlabel $-n$ at 310 650 %
\pinlabel $0$ at 40 705 %
\pinlabel $0$ at 320 705 %
\pinlabel $[1]$ at 50 755 %
\pinlabel $[1]$ at 325 755 %
\pinlabel $[-1]$ at 158 740 %
\pinlabel $[-1]$ at 435 690 %
\pinlabel $k$ at 85 705
\pinlabel $-k$ at 230 705
\pinlabel $k$ at 360 705
\pinlabel $-k$ at 509 705
\endlabellist
    \centering
    \includegraphics[scale=0.7]{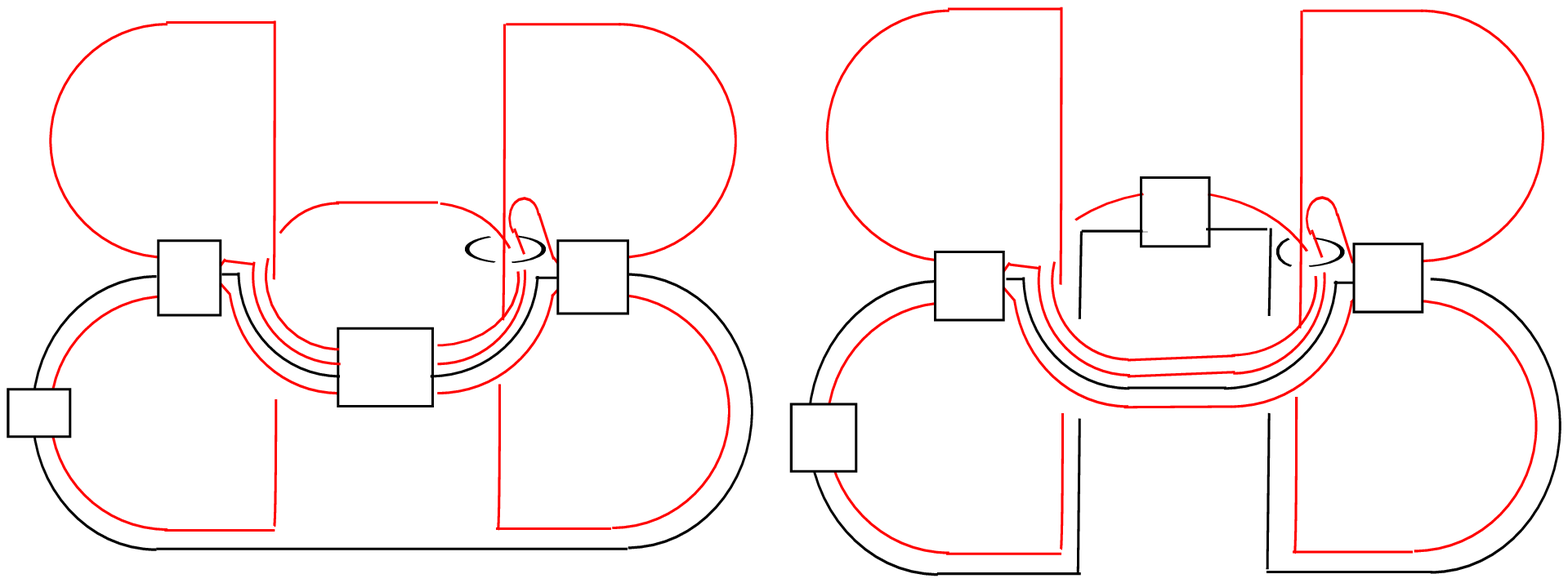}
    \caption{} \label{fig:Gompffig3b}
    \end{figure}

Aim now for the square knot, by setting $k = 1$ and changing the full twists to half twists by a flype.  Figure \ref{fig:Gompffig4} then shows a series of isotopies (clockwise around the figure beginning at the upper left) that simplify the red $[\pm 1]$-framed circles, while dragging along the $0$-framed black circles.

 \begin{figure}[ht!]
  \labellist
\small\hair 2pt
\pinlabel $n$ at 150 365 %
\pinlabel $n$ at 435 370 %
\pinlabel $n$ at 445 180 %
\pinlabel $n$ at 200 175 %
\pinlabel $-n$ at 30 410
\pinlabel $-n$ at 320 415
\pinlabel $-n$ at 40 230
\pinlabel $-n$ at 320 235
\pinlabel $0$ at 35 460
\pinlabel $0$ at 325 460
\pinlabel $0$ at 325 300
\pinlabel $0$ at 40 280
\pinlabel $[1]$ at 45 510
\pinlabel $[1]$ at 340 510
\pinlabel $[1]$ at 355 270
\pinlabel $[1]$ at 70 260
\pinlabel $[-1]$ at 130 470
\pinlabel $[-1]$ at 455 410
\pinlabel $[-1]$ at 400 240
\pinlabel $[-1]$ at 240 220
\pinlabel $0$ at 160 450
\pinlabel $0$ at 430 400
\pinlabel $0$ at 470 325
\pinlabel $0$ at 220 270
\endlabellist
    \centering
    \includegraphics[scale=0.7]{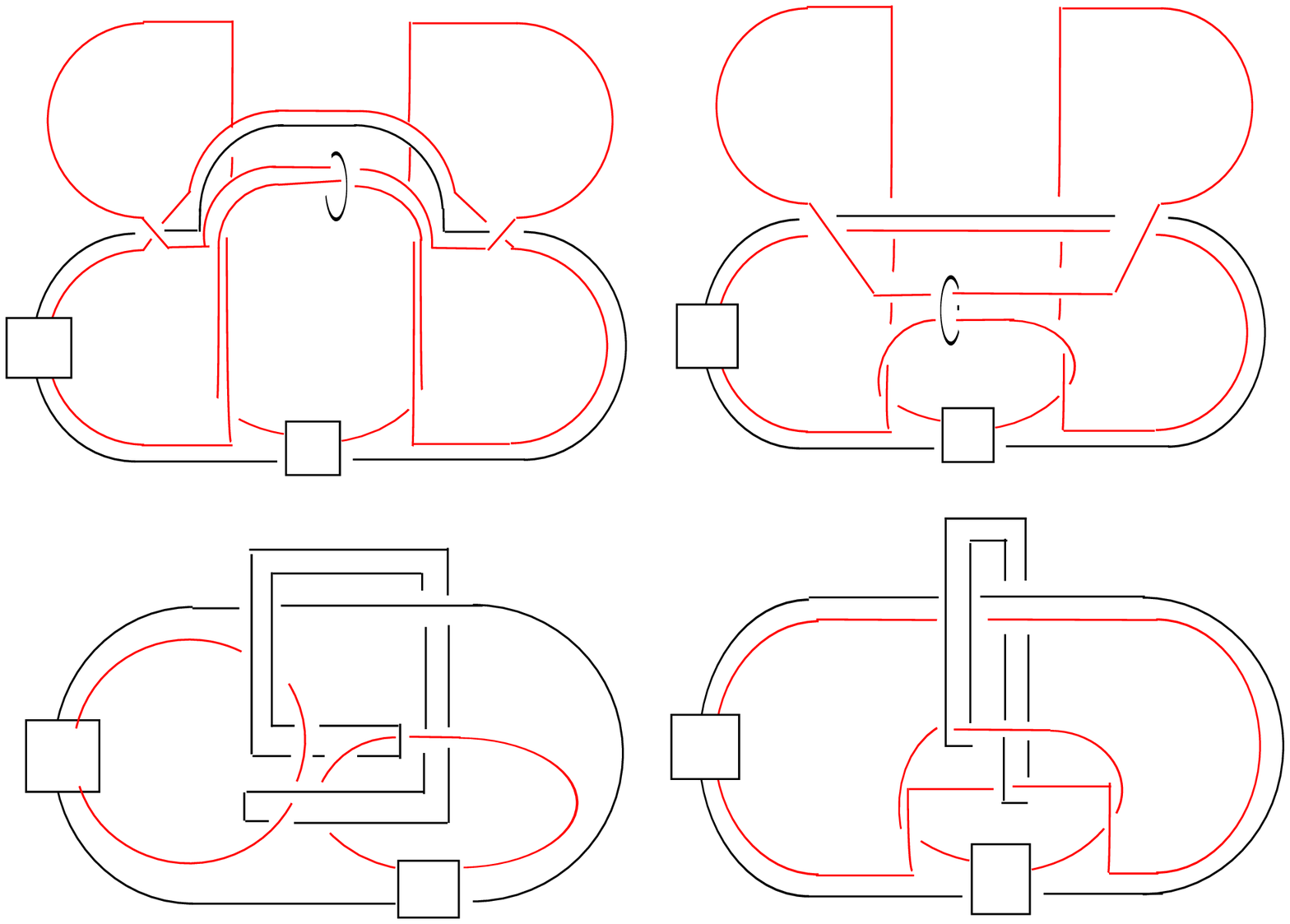}
    \caption{} \label{fig:Gompffig4}
    \end{figure}

Finally, Figure \ref{fig:Gompffig5} shows that the middle $0$-framed component, when the two red $[\pm 1]$-framed components are blown down, becomes the square knot $Q \subset S^3$.  (The other component becomes an interleaved connected sum of two torus knots, $V_n = T_{n, n+1} \# \overline{T_{n, n+1}}$.)

This leaves two natural questions.

     \begin{figure}[ht!]
       \labellist
\small\hair 2pt
     \pinlabel $[1]$ at 75 530
\pinlabel $[-1]$ at 215 460
\endlabellist
    \centering
    \includegraphics[scale=0.7]{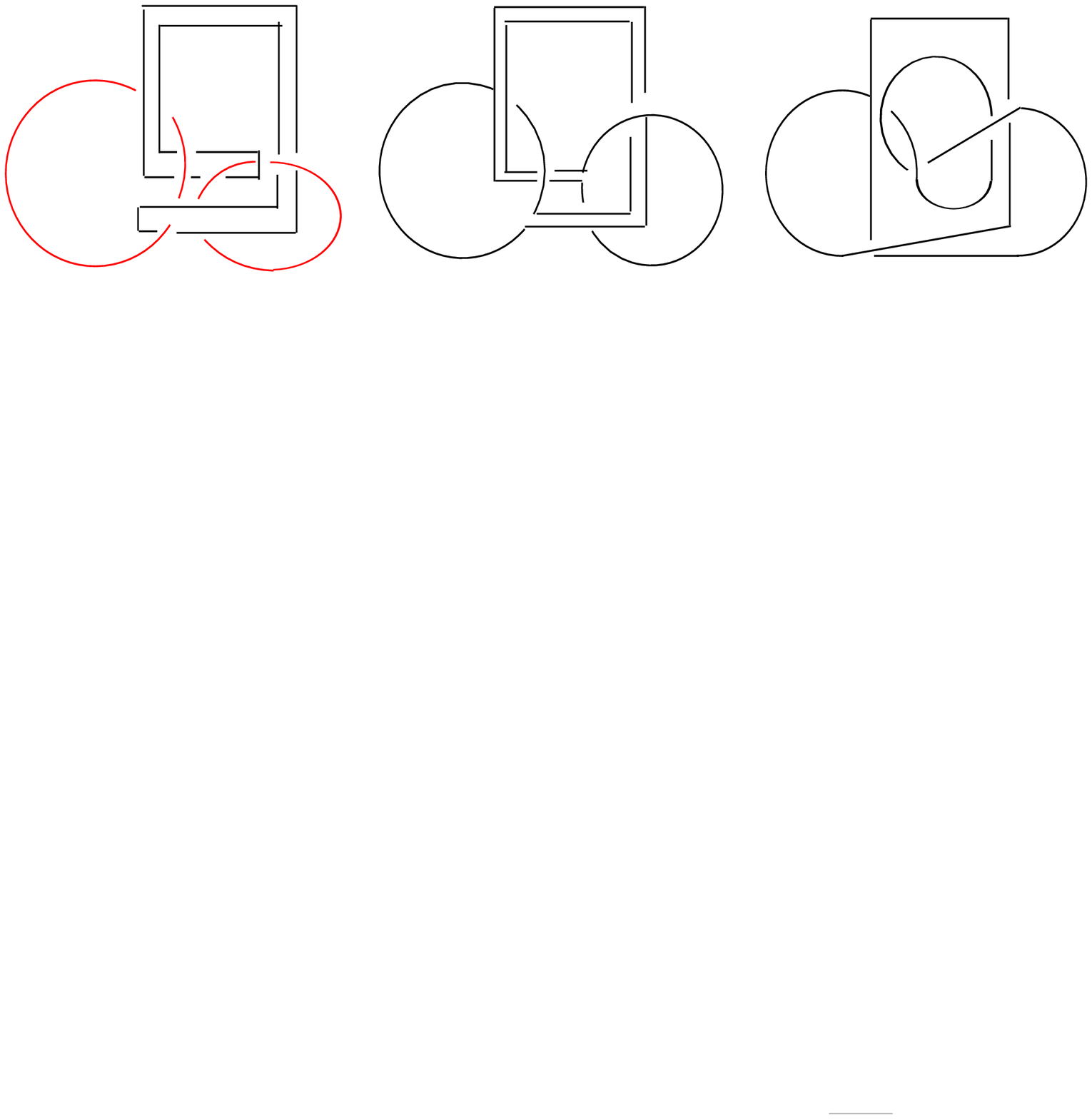}
    \caption{} \label{fig:Gompffig5}
    \end{figure}

\bigskip

{\em Question One:}  As described, $V_n$ does not obviously lie on a Seifert surface for $Q$.  According to Corollary \ref{cor:enumerate}, some handle slides of $V_n$ over $Q$ should alter $V_n$ so that it is one of the easily enumerated curves that do lie on the Seifert surface, in particular it would be among those that are lifts of (half of) the essential simple closed curves in the $4$-punctured sphere $P$.  Which curves in $P$ represent $V_n$ for some $n$?

\bigskip

{\em Question Two:}  Is each $Q \cup V_n, n \geq 3$, a counter-example to Generalized Property R?

\bigskip

This second question is motivated by Figure~\ref{fig:Gompffig1b}. As described in \cite{Go1}, the first diagram of that figure exhibits a simply connected 2-complex, presenting the trivial group as
 $$\langle x, y \; | \; y = w^{-1} x w,\  x^{n+1} = y^n \rangle,$$ 
 where $w$ is some word in $x^{\pm 1},y^{\pm 1}$ depending on $k$ and equal to $yx$ when $k=1$. If the 2-component link $L_{n,k}$ of Figure~\ref{fig:Gompffig3b} (obtained by blowing down the two bracketed circles) can be changed to the unlink by handle slides, then the dual slides in Figure~\ref{fig:Gompffig1b} will trivialize that picture, showing that the above presentation is Andrews-Curtis trivial.  For $k=1$, for example, this is regarded as very unlikely when $n \geq 3$. Since surgery on $L_{n,k}$ is $\#_2(S^1\times S^2)$ by construction, this suggests an affirmative answer to Question Two, which (for any one $n$) would imply:

\bigskip

\begin{conj} \label{conj:not2R} The square knot does not have Property 2R.
\end{conj}

The invariant we have implicitly invoked here can be described in a purely 3-dimensional way.  Suppose $L$ is an $n$-component framed link that satisfies the hypothesis of Generalized Property~R. Then surgery on $L$ yields $\#_n(S^1\times S^2)$, whose fundamental group $G$ is free on $n$ generators. Its basis $\{g_i\}$ is unique up to Nielsen moves. If we pick a meridian of each component of $L$, attached somehow to the base point, we obtain $n$ elements $\{r_i\}$ that normally generate $G$ (since they normally generate the group of the link complement). Thus, we have a presentation $\langle g_1,\cdots,g_n| r_1,\cdots,r_n\rangle$ of the trivial group. Changing our choices in the construction changes the presentation by Andrews-Curtis moves. If we change $L$ by sliding one component $L_i$ over another $L_j$, then $r_j$ ceases to be a meridian, but it again becomes one via the dual slide over a meridian to $L_i$. This is again an Andrews-Curtis move, multiplying $r_j$ by a conjugate of $r_i$. Thus, the link $L$ up to handle slides determines a balanced presentation of the trivial group up to Andrews-Curtis moves.  Unfortunately, there is presently no way to distinguish Andrews-Curtis equivalence classes from each other. When such technology emerges, it should be able to distinguish handle-slide equivalence classes of links satisfying the hypothesis of Generalized Property~R, such as the links $L_{n,k}$ of Figure~\ref{fig:Gompffig3b}.   For a related perspective, see also \cite{CL}.

\bigskip

{\em Remark:} The link $L_{n,k}$ is interesting for another reason. Recall that it is exhibited in Figure~\ref{fig:Gompffig1b} as the two meridians of the 2-handles, which obviously bound disjoint disks in the pictured 4-manifold. (These are the cores of the 2-handles in the dual picture.) Since the 4-manifold is actually diffeomorphic to a 4-ball \cite{Go1}, it follows that $L_{n,k}$ is smoothly slice. Furthermore, each component is ribbon. (A band move turns the small component in  Figure~\ref{fig:Gompffig3b} into a meridian of each of the bracketed circles, becoming an unlink when these are blown down, and the other component is easily seen to be  $T_{n, n+1} \# \overline{T_{n, n+1}}$.) However, there seems to be no reason to expect  $L_{n,k}$  to be a ribbon link. (In general, a link satisfying the hypothesis of Generalized Property R is slice in a possibly exotic 4-ball, whereas the conclusion of Generalized Property R implies it is ribbon.) Band-summing together the components of  $L_{n,k}$ yields a smoothly slice knot that could potentially fail to be ribbon.

\section{Weakening Property nR}

The evidence above suggests that Property 2R, and hence the Generalized Property R Conjecture, may fail for a knot as simple as the square knot.  There are other reasons to be unhappy with this set of conjectures:   For one thing, there is no clear relation between Property nR and Property (n+1)R:  If $K$ has Property nR, there is no reason to conclude that no $(n+1)$-component counterexample to Generalized Property R contains $K$.  After all, Gabai has shown that {\em every knot has Property 1R}.  Conversely, if $K$ does not have Property nR, so there is an $n$-component counterexample $L$ to Generalized Property R, and $K \subset L$, it may well be that even adding a distant $0$-framed unknot to $L$ creates a link that satisfies Generalized Property R.  So $K$ might still satisfy Property (n+1)R.

A four-dimensional perspective on this last possibility is instructive.  (See also \cite[Section 3]{FGMW}. ) Suppose $L$ is an $n$-component link on which surgery gives $\#_{n} (S^{1} \times S^{2})$.  Suppose further that, after adding a distant $0$-framed $r$-component unlink to $L$, the resulting link $L'$  can be reduced to an $(n+r)$-component unlink by handle-slides.  Consider the $4$-manifold $W$ obtained by attaching $2$-handles to $D^4$ via the framed link $L$, then attaching $\natural_{n} (S^{1} \times B^{3})$ to the resulting manifold along their common boundary $\#_{n} (S^{1} \times S^{2})$. Via \cite{LP} we know there is essentially only one way to do this.  The result is a simply-connected (since no $1$-handles are attached) homology $4$-sphere, hence a homotopy $4$-sphere.  If the $2$-handles attached along $L$ can be be slid so that the attaching link is the unlink, this would show that $W \cong S^4$, since it implies that the $2$-handles are exactly canceled by the $3$-handles.  But the same is true if handle-slides convert $L'$ to the unlink: the $4$-manifold $W$ is the same, but attaching first $r$ more trivial $2$-handles and then canceling with $r$ more $3$-handles has no effect on the topology of the associated $4$-manifold.  So from the point of view of $4$-manifolds, this weaker form of Generalized Property R (in which a distant $r$-component unlink may be added to the original link) would suffice.

From the $4$-manifold point of view there is a dual operation which also makes no difference to the topology of the underlying $4$-manifold: adding also pairs of canceling $1$- and $2$- handles.  From the point of view of Kirby calculus, each such pair is conventionally noted in two possible ways (see Section 5.4 of \cite{GS} and Figure \ref{fig:canceling} below):
\bigskip

\begin{itemize}

\item A dumb-bell shaped object.  The ends of the dumb-bell represent $3$-balls on which a $1$-handle is attached; the connecting rod represents the attaching circle for a canceling $2$-handle.

\bigskip

\item A Hopf link.  One component of the link is labeled with a dot and the other is given framing $0$.  The dotted component represents the $1$-handle and the $0$-framed component represents the attaching circle for a canceling $2$-handle.  Call such an object a canceling Hopf pair and call the union of $s$ such pairs, each lying in a disjoint $3$-ball, a {\em set of $s$ canceling Hopf pairs}.
\end{itemize}

\bigskip

(The rules for handle-slides of these dotted components in canceling Hopf pairs are fairly restricted; see \cite{GS}.  For example, they cannot slide over non-dotted components. Note that while we require the 2-handle of a canceling Hopf pair to have framing 0, we can change it to any even framing by sliding the 2-handle over the dotted circle to exploit the nonzero linking number.)

 \begin{figure}[ht!]
  \labellist
\small\hair 2pt
\pinlabel $0$ at 160 650
\pinlabel $0$ at 510 650
 \endlabellist
    \centering
    \includegraphics[scale=0.6]{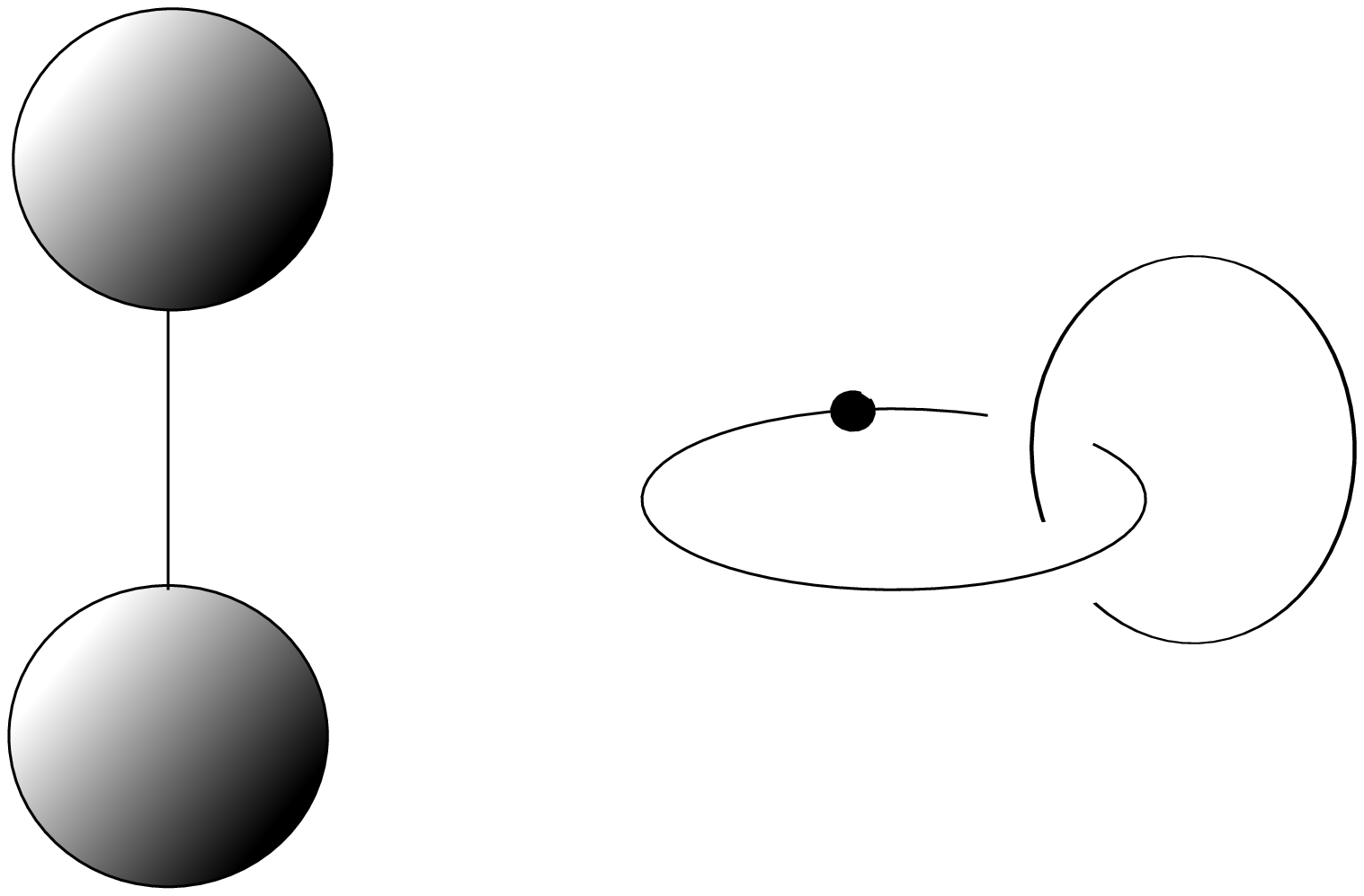}
    \caption{} \label{fig:canceling}
    \end{figure}

The $4$-manifold perspective suggests the following weaker and more awkward version of Generalized Property R:

\begin{conj}[Weak Generalized Property R] \label{conj:weakgenR} Suppose $L$ is a framed link of $n \geq 1$ components in $S^3$, and surgery on $L$ yields $\#_{n} (S^{1} \times S^{2})$.  Then, perhaps after adding a distant $r$-component $0$-framed unlink and a set of $s$ canceling Hopf pairs to $L$, there is a sequence of handle-slides that creates the distant union of an $n+r$ component $0$-framed unlink with a set of $s$ canceling Hopf pairs.
\end{conj}

The move to Weak Generalized Property R destroys the Andrews-Curtis invariant of the previous section, since adding a canceling Hopf pair leaves the surgered manifold, and hence $G$, unchanged, but introduces a new relator that is obviously trivial.  In fact, we will see in the next section that the links $L_{n,k}$ are not counterexamples to this weaker conjecture (even if we set set $r=0$, $s=1$). On the other hand, stabilization by adding a distant unknot to $L$ changes the corresponding presentation by adding a generator $g_{n+1}$ and relator $r_{n+1}=g_{n+1}$, and this preserves the Andrews-Curtis class. Thus, the Hopf pairs seem crucial to the conjecture, whereas it is unclear whether the distant unlink is significant. (Of course, if we allow 1-handles in our original diagram as in Figure~\ref{fig:Gompffig1b}, then the Andrews-Curtis problem seems to make the distant unlink necessary.)

\begin{defin} A knot $K \subset S^3$ has {\bf Weak Property nR} if it does not appear among the components of any $n$-component counterexample to the Weak Generalized Property R conjecture.
\end{defin}

 We have seen that Property nR and Weak Property nR are probably quite different, due to the Hopf pairs. The latter property also exhibits nicer formal behavior due to the distant unlink, in that if $K$ has Weak Property (n+1)R then it has Weak Property nR:  Suppose $K$ appears among the components of an $n$-component framed link $L$ and surgery on $L$ via the specified framing yields $\#_{n} (S^{1} \times S^{2})$. Then the corresponding statement is true for the union $L'$ of $L$ with a distant $0$-framed unknot.  If $K$ has Weak Property (n+1)R then the distant union $L''$ of $L'$ with the $0$-framed $r$-component unlink and a set of $s$ canceling Hopf pairs has a sequence of handle-slides that creates the $0$-framed unlink of $n + r + 1$ components and a set of $s$ canceling Hopf pairs.  But $L''$ can also be viewed as the union of $L$ with the $0$-framed $(r+1)$-component unlink and a set of $s$  canceling Hopf pairs. Thus $K$ satisfies Weak Property~nR.

The Weak Generalized Property R Conjecture is closely related to the Smooth (or PL) 4-Dimensional Poincar\'e Conjecture, that every homotopy 4-sphere is actually diffeomorphic to $S^4$. For a precise statement, we restrict attention to homotopy spheres that admit handle decompositions without 1-handles.

\begin{prop}  The Weak Generalized Property R Conjecture is equivalent to the Smooth 4-Dimensional Poincar\'e Conjecture for homotopy spheres that admit handle decompositions without 1-handles.
\end{prop}

While there are various known ways of constructing potential counterexamples to the Smooth 4-Dimensional Poincar\'e Conjecture, each method is known to produce standard 4-spheres in many special cases. (The most recent developments are \cite{Ak}, \cite{Go2}.) Akbulut's recent work \cite{Ak} has eliminated the only promising potential counterexamples currently known to admit handle decompositions without 1-handles. For 3-dimensional renderings of the full Smooth 4-Dimensional Poincar\'e Conjecture and other related conjectures from 4-manifold theory, see \cite{FGMW}.

\begin{proof}  Suppose $\Sigma$ is a homotopy sphere with no 1-handles and $n$ $2$-handles.  Since there are no $1$-handles, the $2$-handles are attached to some framed $n$-component link $L \subset S^3$ in the boundary of the unique $0$-handle $D^4$ in $\Sigma$. Since $\Sigma$ has Euler characteristic 2, there are $n$ 3-handles attached to the resulting boundary, showing that surgery on $L$ is $\#_n(S^1\times S^2)$.  If the Weak Generalized Property R Conjecture is true, there are some $r, s$ so that when a distant $0$-framed $r$-component unlink and a distant set of $s$ canceling Hopf pairs is added to $L$ (call the result $L'$) then after a series of handle slides, $L'$ becomes the distant union of an $n+r$ component unlink and a set of $s$ canceling Hopf pairs.

To the given handle decomposition of $\Sigma$, add $r$ copies of canceling $2$- and $3$- handle pairs and $s$ copies of canceling $1$- and $2$-handle pairs.  After this change of handle decomposition, the manifold $\Sigma_2$ that is the union of $0$-, $1$- and $2$-handles can be viewed as obtained by the surgery diagram $L' \subset \bdd D^4$.  After a sequence of handle-slides, which preserve the diffeomorphism type of $\Sigma_2$, $L'$ is simplified as above; it can be further simplified by canceling the $1$- and $2$-handle pairs given by the set of $s$ canceling Hopf pairs.  What remains is a handle description of $\Sigma_2$ given by $0$-framed surgery on an unlink of $n+r$ components.

Since $\Sigma$ has Euler characteristic 2, it is obtained by attaching exactly $(n+r)$ $3$-handles and then a single $4$-handle to $\Sigma_2$.  If we view $\Sigma_2$ as obtained by attaching $2$-handles via the $0$-framed $(n+r)$-component unlink, there is an obvious way to attach some set of $(n+r)$ $3$-handles so that they exactly cancel the $2$-handles, creating $S^4$.  It is a theorem of Laudenbach and Poenaru \cite{LP} that, up to handle-slides, there is really only one way to attach $(n+r)$ $3$-handles to $\Sigma_2$.  Hence $\Sigma$ is diffeomorphic to $S^4$ as required.

Conversely, if $0$-framed surgery on an $n$-component link $L \subset S^3$ gives $\#_{n} (S^{1} \times S^{2})$ then a smooth homotopy $4$-sphere $\Sigma$ can be constructed by attaching to the trace of the surgery $n$ $3$-handles and a $4$-handle.  If $\Sigma$ is $S^4$, then standard Cerf theory says one can introduce some canceling $1$- and $2$- handle pairs, plus some canceling $2$- and $3$- handle pairs to the given handle description of $\Sigma$, then slide handles until all handles cancel.  But introducing these canceling handle pairs, when expressed in the language of framed links, constitutes the extra moves that are allowed under Weak Generalized Property R. (To see that the framings in the canceling Hopf pairs can be taken to be 0, first arrange them to be even by choosing the diagram to respect the unique spin structure on $\Sigma$, then slide as necessary to reduce to the 0-framed case.)\end{proof}

\section{How a Hopf pair can help simplify}

With the Weak Generalized Property R Conjecture in mind, return now to the square knot example of Section \ref{sect:nonstandard}.   In \cite{Go1} it is shown that the introduction of a canceling $2, 3$-handle pair does make Figure~\ref{fig:Gompffig1b} equivalent by handle-slides to the corresponding canceling diagram.  In our dualized context, that means that the introduction of a canceling $1, 2$ handle pair (in our terminology, a canceling Hopf pair) should allow $Q \cup V_n$ (or more generally any $L_{n,k}$) to be handle-slid until the result is the union of the unlink and a canceling Hopf pair.  In particular, the square knot examples provide no evidence against the Weak Generalized Property R Conjecture.

It is a bit mysterious how adding a Hopf pair can help to simplify links. Since the dotted circle is not allowed to slide over the other link components, it may seem that any slides over its meridian would only need to be undone to remove extraneous linking with the dotted circle. In this section, we give a general procedure that exploits a nontrivial self-isotopy of the meridian to create families of potential counterexamples to Generalized Property~R. The examples $L_{n,k}$ of Section~\ref{sect:nonstandard} arise as a special case for each fixed $k$.  This suggests that Generalized Property~R probably fails frequently for links constructed by this method, whereas Weak Generalized Property~R does not.

\bigskip

Begin with an oriented 3-manifold $M$ containing a framed link $L=L'\cup L''$.  We will construct an infinite family of framed links in $M$ that may not be handle-slide equivalent, but become so after a Hopf pair is added. (In our main examples, $M$ will be $S^3$ with $L$ a 0-framed 2-component link.) Let $M'$ be the manifold made from $M$ by surgery on $L'$.

Let $\varphi: T^2\to M'$ be a generic immersion that is an embedding on each fiber $S^1\times\{t\}$, $t\in S^1=\R/\zed$.  Denote the image circle $\varphi(S^1\times\{t\})$ by $C_t$. Since $\varphi$ is generic, it is transverse both to $L''$ and to the core of the surgery solid tori in $M'$.  In particular,
\begin{itemize}
\item only finitely many circles $C_{t_i}$  intersect $L''$
\item each such circle $C_{t_i}$ intersects $L''$ in exactly one point
\item the curves $C_{t_i}$ are pairwise disjoint and
\item each $C_{t_i}$ lies in the link complement $M-L' \subset M'$.
\end{itemize}

The immersion $\varphi$ determines an infinite family of framed links $L_n= L'\cup L_n''$ in $M$, $n\in \zed$: For each $t_i$, there is an annulus $A_i$ in $M-L'$ centered on $C_{t_i}$, transverse to $\varphi$ near $S^1\times\{t_i\}$ and intersecting $L''$ in an arc, with $A_i$ oriented so that its positive normal direction corresponds to increasing $t$. See Figure~\ref{fig:Ai}, where the left and right edges of each diagram are glued together. The annuli $A_i$ are unique up to isotopy, and we can assume they are pairwise disjoint. Let $L_n''$ be the framed link obtained from $L''$ by diverting each arc $L''\cap A_i$ so that it spirals $|n|$ times around in $A_i$ (relative to its original position) as in Figure~\ref{fig:Ai}, starting with a right turn iff $n>0$. (Thus $L_0=L$.) We specify the framings by taking them to be tangent to $A_i$ everywhere inside these annuli (before and after) and unchanged outside. Note that $L_n$ depends on our initial choices of how to isotope the curves $C_{t_i}$ off of the surgery solid tori, but changing the choices only modifies $L_n$ by handle slides of $L_n''$ over $L'$.

 \begin{figure}[ht!]
 \labellist
\small\hair 2pt
\pinlabel  $t$ at 50 62
\pinlabel  $A_i$ at 100 95
\pinlabel  $A_i$ at 360 95
\pinlabel  {${\rm Im}\; \varphi$} at 85 20
\pinlabel  $C_t$ at 275 80
\pinlabel  $C_{t_i}$ at 245 62
\pinlabel  $L''$ at 145 130
\pinlabel  $L''_2$ at 420 130
\pinlabel  ${\rm framing}$ [l] at 170 130
\endlabellist
     \centering
    \includegraphics[scale=0.7]{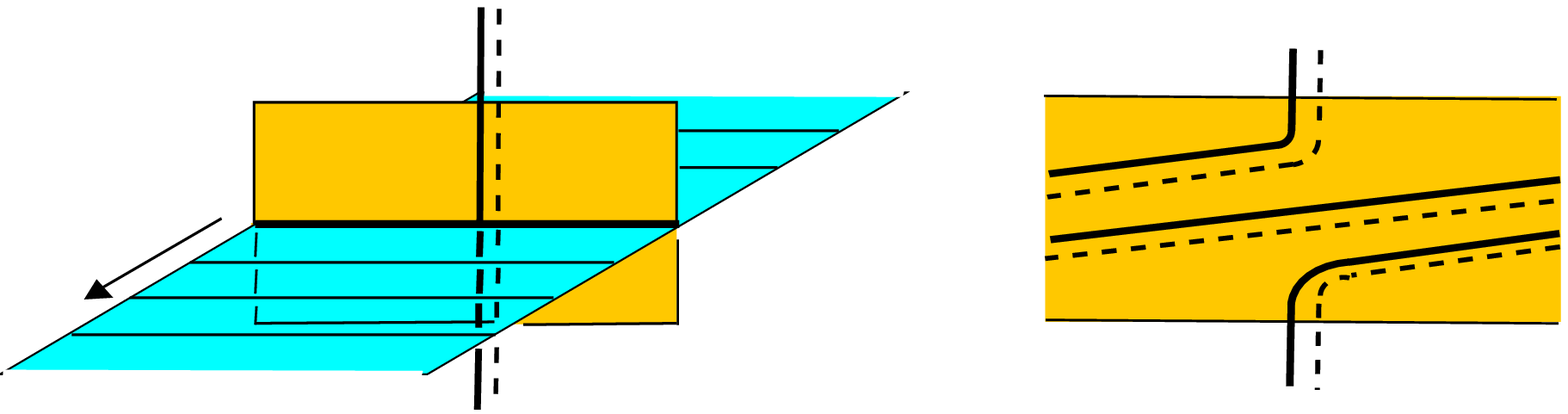}
    \caption{} \label{fig:Ai}
    \end{figure}

In our main examples below, $\varphi$ is an embedding. Whenever this occurs, there is a self-diffeomorphism of $M'$ sending $L''$ onto  $L_n''$, namely the $n$-fold Dehn twist along $\im \varphi$ parallel to the circles $C_t$. This shows directly in this case that the manifolds obtained by surgery on $L_n$ are all diffeomorphic, regardless of $n$. In general, we can obtain more specific information:

\begin{prop} \label{prop:Ln}
For $L\subset M$ and $\varphi$ a generic immersion as above, suppose that $C_0$ is disjoint from all other circles $C_t$  and bounds an embedded disk in $M$ whose algebraic intersection with nearby curves $C_t$ is $\pm 1$. Then all of the framed links $L_n$ become handle-slide equivalent after the addition of a single canceling Hopf pair.
\end{prop}

\noindent Since $C_0$ is nullhomologous in $M$, it has a canonical longitude, determined by any Seifert surface in $M$. Thus, the last hypothesis can be restated by saying that $C_0$ is an unknot in $M$ for which the immersion $\varphi$ determines the $\pm 1$ framing.

To see that the hypotheses are not unreasonable, note that they are satisfied for $M=S^3$, $L'$ empty and $C_0$ a $(1,\pm 1)$-curve in the standard torus $T$. In this case, the links $L''_n$ are obviously all equivalent for any fixed $L''$. (Since $T$ is embedded, the links $L''_n$ are all related as above by Dehn twisting along $T$ in $M'=M=S^3$. These Dehn twists are isotopic to the identity by an isotopy fixing one complementary solid torus of $T$ and twisting the other by both of its circle actions.) The more interesting examples $L_{n,k}$ that arise in \cite{Go1} will be exhibited below. These all arise from the same general procedure: Start with a 0-framed unlink in $S^3$. Slide handles to create a more complicated link $L$ satisfying Generalized Property~R. Locate a suitable immersion $\varphi$, then apply Proposition~\ref{prop:Ln} to create an infinite family $\{L_n\}$ of links that still satisfy the hypothesis of Generalized Property~R, but not necessarily the conclusion for most $n$. One can generalize still further, for example replacing the domain of $\varphi$ by a Klein bottle, although the authors presently have no explicit examples of this.

\begin{proof}[Proof of Proposition~\ref{prop:Ln}]
Since $\varphi$ is an immersion and $C_0$ is disjoint from all other $C_t$, there is a neighborhood of $C_0$ in $M$ intersecting $\im \varphi$ in an annulus. Let $C^*$ be a circle in this neighborhood, parallel to $C_0$ and disjoint from  $\im \varphi$. Thus, it is pushed off of $C_0$ by the $\pm 1$ framing in $M$. (All signs in the proof agree with the one in the statement of the theorem.) Let $L^*_n\subset M$ be the framed link obtained from $L_n$ by adding a $\pm 2$-framed component along $C_0$ and a dotted circle along $C^*$. We will show that $L_n^*$ is obtained from $L_n$ by adding a canceling Hopf pair and sliding handles, and that the links $L_n^*$ are all handle-slide equivalent. (As usual, the dotted circle is not allowed to slide over the other link components.) This will complete the proof.

To reduce $L^*_n$ to $L_n$, slide $C_0$ over the dotted circle $C^*$ using the simplest band connecting the parallel circles. This changes $C_0$ to a 0-framed meridian of $C^*$. (To compute the framing, recall that the dotted circle is automatically 0-framed, and that if $C_0$ and $C^*$ have parallel orientations, we are subtracting handles with linking number $\pm 1$.) The dotted circle $C^*$ bounds a disk $D\subset M$ and now has a 0-framed meridian, but this is not yet a canceling Hopf pair since the link $L$ may still intersect $D$. However, we can easily remove each intersection by sliding the offending strand of $L$ over the meridian. This does not change $L$ but liberates the Hopf pair as required.

To see that the links $L_n^*$ are handle-slide equivalent, we wish to isotope their common component $C$ at $C_0$ around the continuous family $C_t$, returning to the original position at $C_1=C_0$. By construction, $C$ will never meet the dotted circle at $C^*$, but it will meet components of $L_n$. Since the isotopy lies in $M'$, each encounter with $L'$ consists of a push across a surgery solid torus in $M'$, which is precisely a handle slide in $M$ of $C$ over a component of $L'$. On the other hand, $L''_n$ transversely intersects the path of $C$, once at each annulus $A_i$, so we must somehow reverse the crossing of $C$ with $L_n''$ there. Since the framing of $C$ has $\pm 1$ twist relative to $\varphi$, and hence to $A_i$, we can reverse the crossing by sliding $L_n''$ over $C$. This is shown by Figure~\ref{fig:cross} (or its mirror image if the sign is $-1$), where  $A_i$ is drawn as in Figure~\ref{fig:Ai}, but after a diffeomorphism so that $L_n''$ appears as a vertical line. (To see that the framing transforms as drawn, note that the right twist in the framing of the curve parallel to $C$ cancels the left twist introduced by the self-crossing added to $L_n''$ by the band-sum.) When $C$ returns to its original position at $C_1=C_0$, the link has been transformed to $L_{n\pm1}^*$ by handle slides, completing the proof.
\end{proof}

 \begin{figure}[ht]
 \labellist
\small\hair 2pt
\pinlabel  $A_i$ at 20 370
\pinlabel  $A_i$ at 20 200
\pinlabel  $A_i$ at 20 33
\pinlabel  $C$ at 65 390
\pinlabel  $C$ at 65 210
\pinlabel  $C$ at 65 65
\pinlabel  {${\rm pushed\; off} \; C$} at 55 450
\pinlabel  $L''_n$ at 105 480
\pinlabel  $L''_{n+1}$ at 100 147
\endlabellist
     \centering
    \includegraphics[scale=0.7]{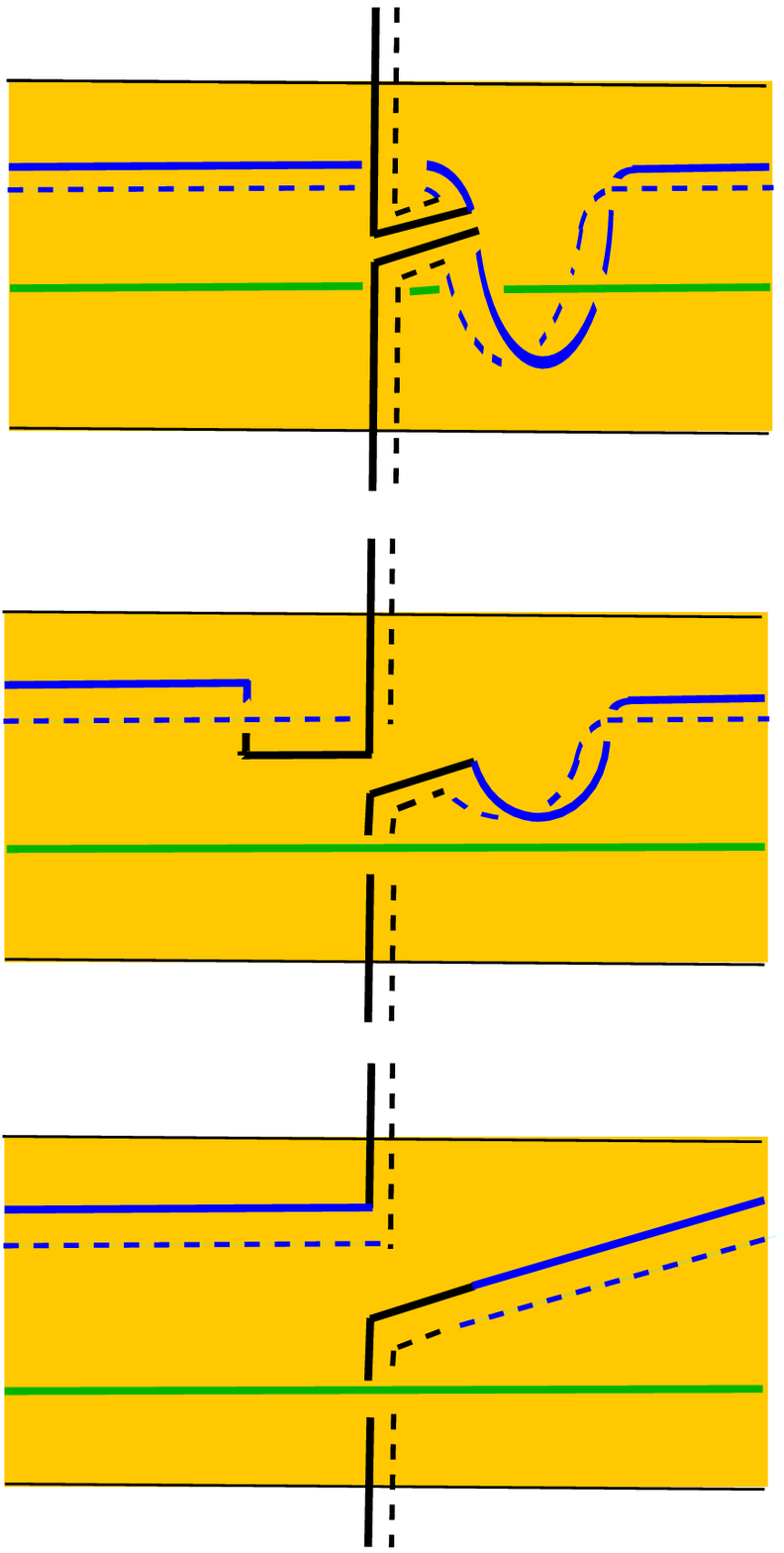}
    \caption{} \label{fig:cross}
    \end{figure}

The examples $L_{n,k}$ from Section~\ref{sect:nonstandard} form a family as in Proposition~\ref{prop:Ln} for each fixed $k$, generated by the handle-slide trivial case $n=0$. Figure \ref{fig:gompf6} shows the starting point.  It is basically the right side of Figure \ref{fig:Gompffig3b} with both a twist-box and one of the 0-framed links moved to a more symmetric position.

 \begin{figure}[ht]
  \labellist
\small\hair 2pt
\pinlabel  $k$ at 80 130
\pinlabel  $-k$ at 275 130
\pinlabel  $n$ at 170 160
\pinlabel  $-n$ at 170 50
\pinlabel  $[1]$ at 20 245
\pinlabel  $[-1]$ at 210 180
\pinlabel  $0$ at 175 115
\pinlabel  $0$ at 25 130
\endlabellist
     \centering
    \includegraphics[scale=0.7]{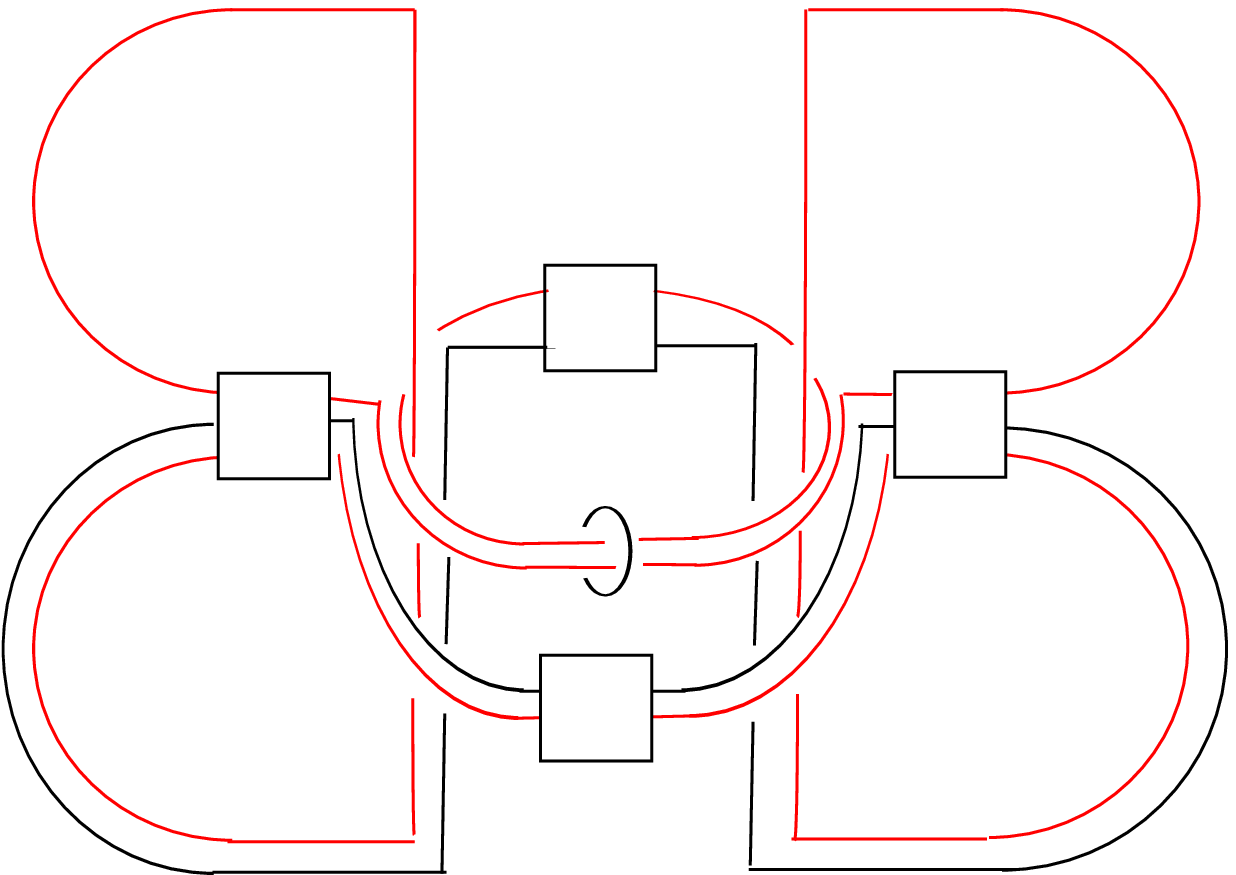}
    \caption{} \label{fig:gompf6}
    \end{figure}

Let $M$ be the manifold diffeomorphic to $S^3$ obtained by surgery on the bracketed (red) circles. (Recall that these form an unlink in $S^3$.) Let $L'$ and $L''_n$ be the small and large 0-framed circles, respectively. There is a vertical plane, perpendicular to the page, that bisects the picture.
This plane bisects both $\pm n$ twist boxes and contains $L'$. We interpret this plane as a 2-sphere in the background $S^3$. Surgery on $L'$ splits the sphere into a disjoint pair of spheres. Neither of these lies in $M'$ since each is punctured by both bracketed circles. However, we can remove the punctures by tubing the spheres together along both bracketed curves, obtaining an embedded torus in $M'$. The fibers $C_t$ of this torus include a meridian of each bracketed circle, and the torus induces the 0-framing on these relative to the background $S^3$. To see these in $M$, ignore the link $L_n$ and unwind the bracketed link. Blowing down reveals the 0-framed meridian of the $[\mp 1]$-framed curve to be a $\pm 1$-framed unknot in $M$ (see \cite[Figure 5.18]{GS}). Thus, we can choose either meridian to be $C_0$, with framing $\pm 2$ in $M$, i.e., $\pm 1$ in $S^3$. (In either case, the framing in $S^3$ is preserved as we travel around the torus, so it becomes 0 in $M$ when we reach the other meridian, showing that we cannot get off the ride in the middle.) Since $L_n''$ intersects the torus at two points, at the $\pm n$-twist boxes, it is easy to verify that the procedure turns $L_n^*$ into $L_{n\pm 1}^*$, with the sign depending both on the choice of $C_0$ and the direction of motion around the torus. It is instructive to view the procedure explicitly as a sequence of four handle slides in Figure~\ref{fig:gompf6}.

It is not completely obvious how to trivialize the link by handle slides when $n=0$. In the $k=1$ case, Figures \ref{fig:Gompffig4} and \ref{fig:Gompffig5} isotope the link to Figure \ref{fig:squareknot}, solving the problem with a single handle slide. For general $k$, it is helpful to begin from a somewhat simpler picture of $L_{n,k}$, Figure~\ref{fig:example}, which is obtained from Figure \ref{fig:gompf6} by an isotopy that rotates both $\pm k$ twist boxes 180 degrees outward about vertical axes. Now set $n = 0$ and consider the two handle slides in Figure~\ref{fig:example2} (following the arrows, framed by the plane of the paper).  The slides first complicate $L'$, but after an isotopy shrinking the strands shown in green, $L'$ simplifies,  so that a slide over the $[-1]$-framed circle changes it to a meridian of the $[+1]$-framed circle. The $\pm k$-twist boxes then cancel, showing that the link is trivial.

\begin{figure}[ht]
     \labellist
\small\hair 2pt
\pinlabel  $k$ at 75 150
\pinlabel  $-k$ at 275 150
\pinlabel  $n$ at 170 100
\pinlabel  $-n$ at 170 20
\pinlabel  $[1]$ at 25 220
\pinlabel  $[-1]$ at 220 220
\pinlabel  $0$ at 180 200
\pinlabel  $0$ at 25 100
\endlabellist
 \centering
    \includegraphics[scale=0.7]{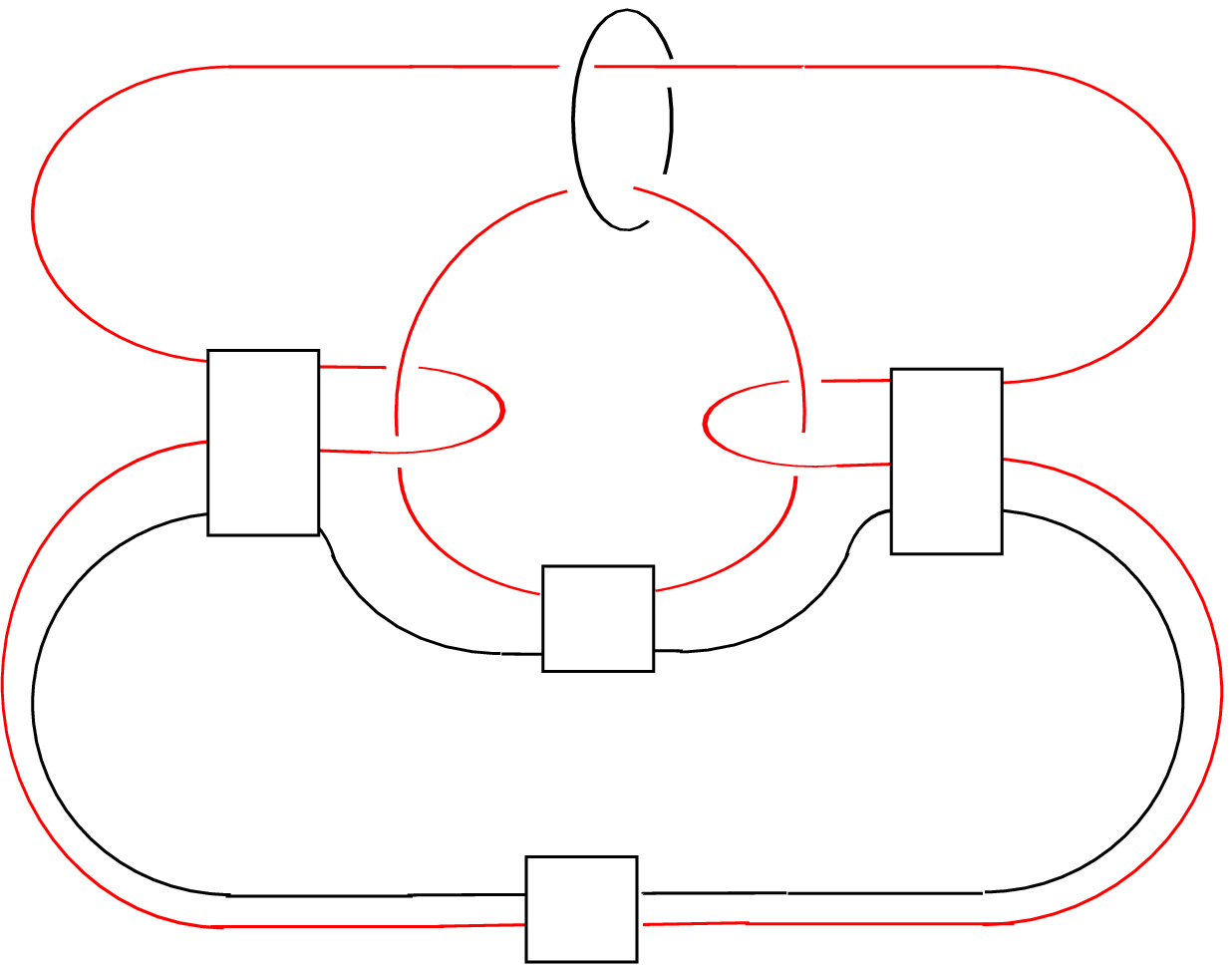}
    \caption{} \label{fig:example}
    \end{figure}

    \begin{figure}[ht]
     \labellist
\small\hair 2pt
\pinlabel  $0$ at 105 433
\pinlabel  $0$ at 105 387
\endlabellist
     \centering
    \includegraphics[scale=0.7]{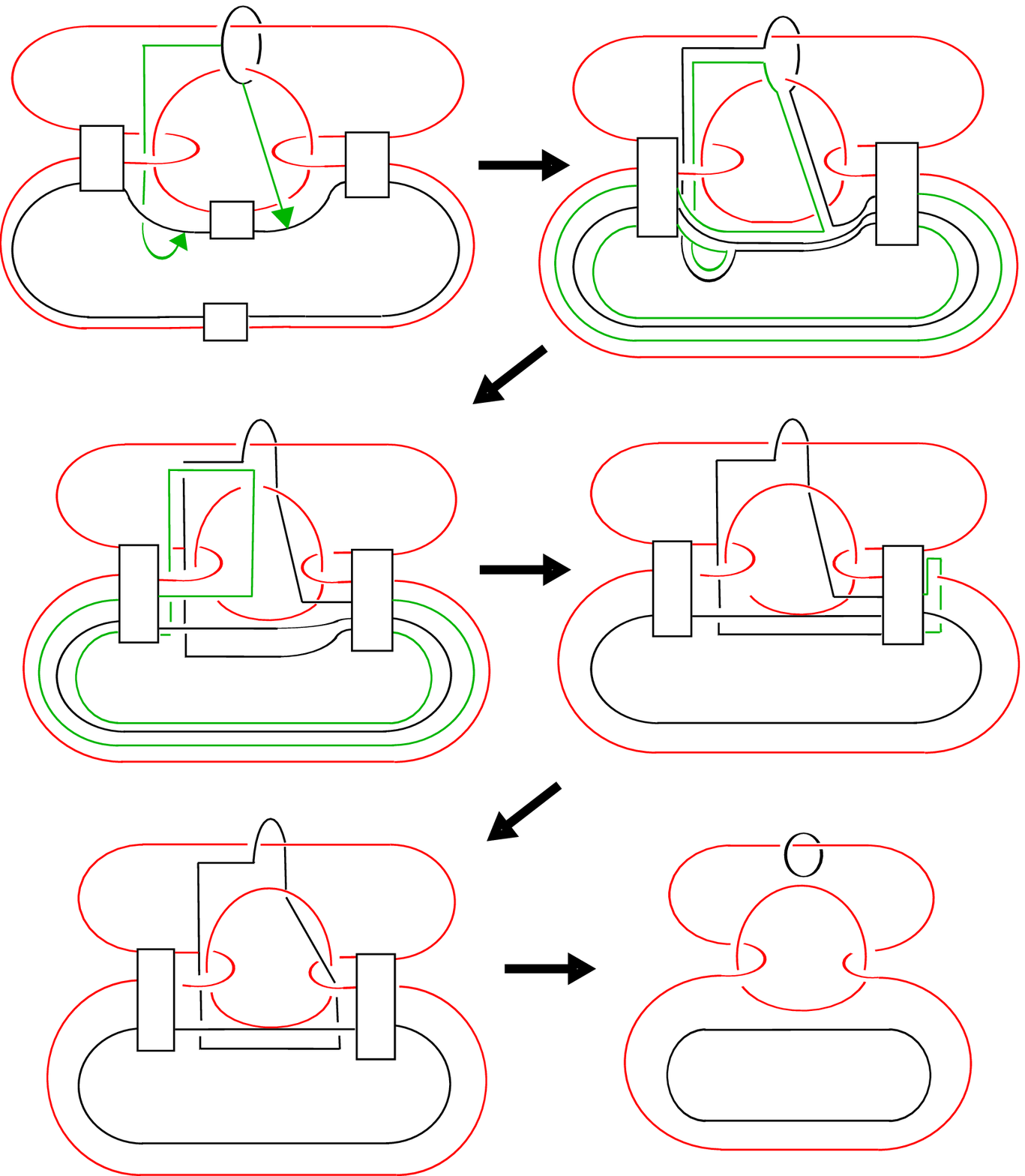}
    \caption{} \label{fig:example2}
    \end{figure}

This construction was derived from the dual construction in \cite{Go1}, showing that the relevant homotopy balls are standard by introducing a canceling 2-handle/3-handle pair. While it was not easy to explicitly dualize the construction, in the end the two constructions appear rather similar. The construction in \cite{Go1} also involves introducing a new link component and moving it around a torus. In that case, the new component was a 0-framed unknot, the attaching circle of the new 2-handle.  Akbulut's recent proof \cite{Ak} that infinitely many Cappell-Shaneson homotopy 4-spheres are standard again relies on such a trick. In each case, the key seems to be an embedded torus with trivial
normal bundle, to which a 2-handle is attached with framing $\pm 1$, forming
a {\em fishtail neighborhood}. This neighborhood has self-diffeomorphisms
that can undo certain cut-and-paste operations on the 4-manifold, as we implicitly
saw in the case of Proposition~\ref{prop:Ln} with $\varphi$ an embedding.
While this procedure has been known in a different context for at least
three decades, underlying the proof that simply elliptic surfaces with fixed $b_2$ are
determined by their multiplicities, it still appears to be underused. In
fact, one can obtain a simpler and more general proof that many Cappell-Shaneson
4-spheres are standard by going back to their original definition and
directly locating fishtail neighborhoods \cite{Go2}.

\end{document}